\documentclass[10pt, journal]{IEEEtran}

\usepackage{epsfig,color,amsmath,cite}
\usepackage{amsthm} 
\usepackage{amsmath}    
\IEEEoverridecommandlockouts
\usepackage{empheq}
\usepackage{bm}
\usepackage{epstopdf}
\usepackage{amssymb}
\usepackage{url}
\usepackage{enumitem} 
\usepackage{multirow}
\usepackage{hhline}
\usepackage{booktabs}
\usepackage{mathtools}
\usepackage{makecell}
\usepackage[linesnumbered,boxed,commentsnumbered,ruled,vlined,longend]{algorithm2e}
\usepackage{comment}

\DeclareMathOperator*{\minimize}{minimize}

\DeclareMathOperator*{\subjectto}{subject\ to}

\DeclareMathOperator{\Blkdiag}{Blkdiag}
\DeclareMathOperator{\Diag}{Diag}
\makeatother
\DeclareMathAlphabet\mathbfcal{OMS}{cmsy}{b}{n}

\newtheorem{theorem}{Theorem}

\newtheorem{mylem}{Lemma}

\newtheorem{myrem}{Remark}
\newtheorem{asmp}{Assumption}

\newtheorem{myprs}{Proposition}



\usepackage{stackengine}

\newcommand{\bmat}[1]{\begin{bmatrix} #1 \end{bmatrix}}

\newcommand{\m}{\boldsymbol}
\allowdisplaybreaks[4]
\usepackage[colorlinks = true,
linkcolor = blue,
urlcolor  = blue,
citecolor = blue,
anchorcolor = blue]{hyperref}

\newcommand{\mc}[1]{\mathcal{#1}}
\newcommand{\mbb}[1]{\mathbb{#1}}
\newcommand{\mr}[1]{\mathrm{#1}}
\usepackage[framemethod=TikZ]{mdframed}
\mdfdefinestyle{MyFrame}{%
	linecolor=black,
	outerlinewidth=1.25pt,
	roundcorner=1.25pt,
	innerrightmargin=5pt,
	innerleftmargin=5pt,}
	
\usepackage[noabbrev]{cleveref}

\usepackage{mathtools}

\DeclarePairedDelimiter\abs{\lvert}{\rvert}%
\DeclarePairedDelimiter\norm{\lVert}{\rVert}%

\makeatletter
\let\oldabs\abs
\def\abs{\@ifstar{\oldabs}{\oldabs*}}
\let\oldnorm\norm
\def\norm{\@ifstar{\oldnorm}{\oldnorm*}}
\makeatother

\usepackage[english]{babel}
\usepackage[utf8]{inputenc}
\usepackage[super]{nth}

\usepackage{graphicx}
\usepackage{float}
\usepackage[caption = false]{subfig}

\usepackage{array}
\usepackage{threeparttable}

\usepackage[english]{babel}
\usepackage[utf8]{inputenc}
\usepackage[super]{nth}

\RequirePackage{filecontents}

\def\newqed{{\null\nobreak\hfill\color{black}\ensuremath{\blacksquare}}}

\SetKwRepeat{Do}{do}{while}%

\newcommand*\widefbox[1]{\fbox{\hspace{0em}#1\hspace{0em}}}

\usepackage{lipsum}

\title{\centering \Huge {Load and Renewable-Following Control of Linearization-Free Differential Algebraic Equation Power System Models}}

\author{Sebastian A. Nugroh$\text{o}^\star$, \thanks{
		$^\star$Cummins Technical Center, Cummins Inc., 1900 McKinley Ave., Columbus, IN 47201 (sebastian.nugroho@cummins.com).} \textit{IEEE, Member} and Ahmad F. Tah$\text{a}^\dagger$,\thanks{
		$^\dagger$Departments of Civil \& Environmental Engineering and Electrical and Computer Engineering, Vanderbilt University, 2201 West End Ave., Nashville, TN 37235 (ahmad.taha@vanderbilt.edu).} \textit{IEEE, Member}\vspace{-0.7cm}
	\thanks{
	This work is supported by National Science Foundation under Grants 2151571 and 2152450.}
}

\begin{document}

\newdimen\origiwspc%
\newdimen\origiwstr%
\origiwspc=\fontdimen2\font
\origiwstr=\fontdimen3\font

\fontdimen2\font=0.64ex

\maketitle


\begin{abstract}
Electromechanical transients in power networks are mostly caused by a mismatch between power consumption and production, causing generators to deviate from the nominal frequency. To that end, feedback control algorithms have been designed to perform frequency and load/renewables-following control.
In particular, the literature addressed a plethora of grid- and frequency-control challenges with a focus on linearized, differential equation models whereby algebraic constraints (i.e., power flows) are eliminated. This is in contrast with the more realistic nonlinear differential algebraic equation (NDAE) models.
Yet, as grids are increasingly pushed to their limits via intermittent renewables and varying loads, their physical states risk escaping operating regions due to either a poor prediction or sudden changes in renewables or demands---deeming a feedback controller based on a linearization point virtually unusable. In lieu of linearized differential equation models, the objective of this paper is to design a simple, purely decentralized, linearization-free, feedback control law for NDAE models of power networks. The aim of such a controller is to primarily stabilize frequency oscillations after a significant, unknown disturbance in renewables or loads. Although the controller design involves advanced NDAE system theory, the controller itself is as simple as a decentralized proportional or linear quadratic regulator in its implementation. Case studies demonstrate that the proposed controller is able to stabilize dynamic and algebraic states under significant disturbances.
\end{abstract}


\vspace{-0.4cm}

\section{Introduction}
\IEEEPARstart{O}{ver} the years, the trends of global electricity generation have been shifting from fuel-based conventional generators to a mix of such types with fuel-free renewable energy resources such as wind and PV solar farms. Nowadays, renewable energy sources contribute around $21\%$ of the total generated electricity in the U.S. and it is projected that their contribution will double to $42\%$ by 2050 \cite{eia2021}. Albeit the increasing penetration of renewables in bulk power systems plays a vital role in mitigating climate change \cite{inbook2012}, it unfortunately presents a major challenge in power systems operation due to the intermittent and uncertain nature of renewables and loads. This challenge is met by two 
main goals in power systems control which are \textit{(i)} maintaining the balance between power supply and demand while also \textit{(ii)} preserving the systems-wide frequency \cite{Zhao2012Frequency}. Both objectives are essential to achieve successful power systems operation as significant power imbalance and large frequency deviation can provide adverse impacts which can eventually result in system collapse  \cite{osti_885842}.    

The increasing penetration of renewables makes the aforementioned tasks to be remarkably difficult to achieve, and with that in mind, this paper is dedicated to addressing the problem of \textit{load and renewable-following control} (LRFC), which focus is preserving the power balance and system's frequency against unpredictable behavior of power demand and renewables. This problem is closely related to the \textit{load following control} (LFC), in which power imbalance and frequency deviations are mainly only attributed to changes in power demand only \cite{XiaofengYu2004}.
There exist numerous methods to address the problem pertaining to LFC. In multi-area power networks, automatic generation control (AGC) is a secondary, inter-area control architecture which purpose is to regulate the network's frequency and interchange of power flow \cite{Marinovici2013}. 
Other than AGC, many proportional-integral-derivative (PID)-based controllers have also been proposed in the literature. 
Although PID is known for its simplicity, it unfortunately requires rigorous tuning and the results based on conventional approaches are often not generally robust \cite{HOTE2018604}.   
 
The shortcomings of conventional AGC and PID controllers motivate the development of advanced control techniques, particularly for power network applications. The advancement of convex optimization theory as well as computational method facilitates the design of linear matrix inequalities (LMIs)-based stabilization. Advanced control strategies for power networks---albeit are not limited solely for LFC---can be generally categorized into \textit{(a)} unified, wide-area control and \textit{(b)} localized, decentralized control frameworks.
Related to the wide-area control, the authors in \cite{Chow2012}, \cite{Zolotas2007}, and \cite{Jain2015}  respectively employ the adaptive control, linear quadratic Gaussian control, and model predictive control (MPC) frameworks to minimize power oscillations and improve damping between multiple areas. Since these methods result in centralized control laws that may not be suitable for large-scale networks, an optimization based method is developed in \cite{Dorfler2014} to synthesize optimal control policies with sparse stabilizing controller gains. 

The study \cite{Bazrafshan2019coupling} combines the optimal power flow problem with LFC using the linear quadratic regulator (LQR). Recently, a method developed using the notion of $\mathcal{L}_{\infty}$ stability is proposed in \cite{Taha2019TCNS} to implement a robust control architecture for LRFC in power systems. The behavior of power networks with respect to the increasing penetration of distributed energy resources (DERs) including renewables is investigated in \cite{Sadamoto2019}, where it is revealed that the increasing number of DERs connected to the network can reduce the system's stability.
All of the aforementioned studies rely on \textit{linearized} ordinary differential equation (ODE) models of power networks. The drawbacks of this approach are:  \textit{(i)} the linearization and controller synthesis need to be performed periodically while \textit{(ii)} the resulting control law can only stabilize the system in a small operating region.  

For the decentralized grid control architecture, the works in \cite{siljak2002robust,Elloumi2002} pioneer the design of robust decentralized stabilization for interconnected multi-machine power networks modeled as \textit{nonlinear} ODEs. The underlying concept behind this approach is to treat the nonlinearities of the system as a source of uncertainty and as such, provided that these nonlinearities are quadratically bounded, a linear state feedback control gain can be synthesized by solving convex optimization problems. This idea has been utilized in \cite{Zecevic2004,BEFEKADU200591} and later on is extended to enhance power networks' transient dynamics \cite{Marinovici2013} and tackle parametric uncertainties via $\mathcal{H}_{\infty}$ control \cite{WANG2015155}. In addition to this, a decentralized control based on the LQR for improving small signal stability and providing sufficient damping is proposed in \cite{Singh2016}. Albeit the methods proposed in \cite{siljak2002robust,Elloumi2002,Marinovici2013} are not relying on any linearization either, they \textit{(i)} only consider active power transfer between generators and loads, \textit{(ii)} the model assumes a reduced network (generator buses only), and \textit{(iii)} the disturbances due to renewables uncertainty are not considered. 

To circumvent the limitations of these approaches, efforts have been made recently to study the properties as well as the stability of power systems based on their differential-algebraic equation (DAE) models. For instance, \cite{Gross2014} studies the structural properties of the \textit{linearized} DAE model of power networks---this is extended in \cite{GRO201612} to include higher order generator dynamics. Utilizing the model presented in \cite{Gross2014}, the author in \cite{Datta2020} presents a condition to determine the small signal stability of power networks. Moreover, the problem of characterizing topological changes in linear DAE systems is investigated in \cite{PATIL2019280}. A data-driven MPC for linear DAE power system models is proposed in \cite{Schmitz2022} for frequency regulation purposes. The main advantage of using a DAE representation of power networks relative to an ODE is that the behavior of the network's dynamics can be tightly linked with the network's topology and power flow equations. Besides, if nonlinear DAE (NDAE) models are used, the dynamical behavior of the system can be studied across wider operating regions while regulating both the algebraic and dynamic variables in a power system.

Motivated by the drawbacks existing in previous studies, a novel approach for LRFC is presented in this paper by leveraging the classical NDAE models of power networks. The LRFC is derived based on a more comprehensive \nth{4}-order generator dynamic model, complete with generator's complex power and power balance equations. To the best of our knowledge, this is the first attempt to provide a secondary control based on the NDAE models of multi-machine power networks especially for LRFC. The proposed control strategy is intended to maintain the network's frequency, which variability is attributed to a sudden change in power demand and power produced by renewable energy resources. 
The paper's  contributions are threefold:
\vspace{-0.05cm}
\begin{itemize}[leftmargin=*]
	\item The introduction of a new state feedback control framework for LRFC using a detailed, high-order NDAE model of power networks. The proposed LRFC strategy does \textit{not} require any linearization and as such, its gain-computation is not linked to any operating point. Moreover, the resulting state feedback gain matrix has a purely decentralized structure. This \textit{(i)} improves the practicality of the proposed LRFC especially for larger networks while \textit{(ii)} eliminates the need for an optimization strategy to \textit{sparsify} the controller's structure. 
	\item The development of a convex optimization-based approach for the stabilization of NDAEs. Although the stability of NDAEs has been studied in the literature for quite some time (for example, see references \cite{Lu2003stability,Franco2019,Franco2020}), approaches for the stabilization of NDAEs based on LMIs are, unfortunately, still lacking. Hence, we propose herein a computationally friendly approach for the stabilization of NDAEs based on a simple state feedback control policy using LMIs. 
	\item We showcase the effectiveness and performance of the proposed approach to perform LRFC, where it is compared with AGC and LQR-based control (see \cite{Nugroho2022Vintage}  for a version where we also compare the proposed approach with $\mathcal{H}_{\infty}$ control). Numerical test results indicate the superiority of our approach to performing LRFC relative to LQR and AGC, since it can maintain the network's frequency and power balance subject to a relatively large step disturbance. 
\end{itemize}  
\vspace{-0.05cm}
The remainder of the paper is organized as follows. Section \ref{sec:modeling} presents the semi-explicit, NDAE models of power networks while Section \ref{sec:control_design} discusses the design of the proposed state feedback control strategy for the stabilization of NDAEs, especially for LRFC. Thorough numerical studies are provided in Section \ref{sec:numerical_studies} where the results are discussed accordingly. Finally, the paper is concluded in Section \ref{sec:conclusion}. 

\noindent {\textbf{Notation.}} 
The notation $\m 1$ denotes a column vector with elements of $1$ while the notation $\m I$ and $\m O$ represent the identity and zero matrices of appropriate dimensions.
The notations $\mathbb{R}^n$ and $\mathbb{R}^{p\times q}$ denote the sets of row vectors with $n$ elements and matrices with size $p$-by-$q$ with elements in $\mathbb{R}$. {The sets of $n$-dimensional positive definite matrices and  positive real numbers are denoted by $\mathbb{S}^n_{++}$ and $\mathbb{R}_{++}$. The $2$-norm of $\m x\in\mbb{R}^n$ is equal to $\norm{\m x}_2 := \sqrt{x_1^2 + \cdots + x^2_n}$.}
The operators $\mathrm{Blkdiag}(\cdot)$ constructs a block diagonal matrix, $\mathrm{Diag}(\cdot)$ constructs a diagonal matrix from a vector, $\oslash$ denotes the Hadamard division, and $\odot$ denotes the Hadamard product. 
The symbol $*$ represents symmetric entries in symmetric matrices. 

\vspace{-0.3cm}
\begin{table}[t!]
	\footnotesize	\renewcommand{\arraystretch}{1.3}
	\caption{Description of important notations used in this paper.}
	\label{tab:notation}
	\vspace{-0.2cm}
	\centering
	\begin{tabular}{|c|c|}
		\hline
		\textbf{Notation} & \textbf{Description}\\
		\hline
		\hline
		\hspace{-0.1cm}$\mathcal{N}:=\{1,2,\ldots,N\}$ & \hspace{-0.1cm} set of nodes (buses) \\
		\hline
		\hspace{-0.1cm}$\mathcal{E}\subseteq \mathcal{N}\times\mathcal{N}$ & \hspace{-0.1cm} set of edges (links) \\
		\hline
		\hspace{-0.1cm}{$\mathcal{G}\subseteq \mathcal{N}$, $\abs{\mathcal{G}} = G$} & \hspace{-0.1cm} set of generator buses \\
		\hline
		\hspace{-0.1cm}{$\mathcal{R}\subseteq \mathcal{N}$, $\abs{\mathcal{R}} = R$} & \hspace{-0.1cm} set of buses with renewables\\
		\hline
		\hspace{-0.1cm}{$\mathcal{L}\subseteq \mathcal{N}$, $\abs{\mathcal{L}} = L$} & \hspace{-0.1cm} set of load buses \\
		\hline
		\hspace{-0.1cm}{$\mathcal{U}\subseteq \mathcal{N}$, $\abs{\mathcal{U}} = U$} & \hspace{-0.1cm} {set of non-unit buses} \\
		\hline
		\hspace{-0.1cm}$\delta_{i}:=\delta_i(t)$ & \hspace{-0.1cm} generator rotor angle ($\mathrm{rad}$) \\
		\hline
		\hspace{-0.1cm}$\omega_{i}:=\omega_i(t)$ & \hspace{-0.1cm} generator rotor speed ($\mathrm{rad/sec}$) \\
		\hline
		\hspace{-0.1cm}$E'_{i}:=E'_{i}(t)$ & \hspace{-0.1cm} generator transient voltage ($\mathrm{pu}$) \\
		\hline
		\hspace{-0.1cm}$T_{\mr{M}i} := T_{\mr{M}i}(t)$ & \hspace{-0.1cm} generator mechanical input torque ($\mathrm{pu}$)\\
		\hline
		\hspace{-0.1cm}$E_{\mr{fd}i} := E_{\mr{fd}i}(t)$ & \hspace{-0.1cm} generator internal field voltage  ($\mathrm{pu}$) \\
		\hline
		\hspace{-0.1cm}$T_{\mr{r}i} := T_{\mr{r}i}(t)$ & \hspace{-0.1cm} governor reference signal ($\mathrm{pu}$)\\
		\hline
		\hspace{-0.1cm}$M_i$ & \hspace{-0.1cm} rotor inertia constant ($\mr{pu} \times \mr{sec}^2$) \\
		\hline
		\hspace{-0.1cm}$D_i$ & \hspace{-0.1cm} damping coefficient ($\mr{pu} \times \mr{sec}$)  \\
		\hline
		\hspace{-0.1cm}$x_{\mr{d}i}$& \hspace{-0.1cm} direct-axis synchronous reactance ($\mr{pu}$) \\
		\hline
		\hspace{-0.1cm}$x_{\mr{q}i}$& \hspace{-0.1cm} direct-axis synchronous reactance ($\mr{pu}$) \\
		\hline
		\hspace{-0.1cm}$x'_{\mr{d}i}$ & \hspace{-0.1cm} direct-axis transient reactance ($\mr{pu}$) \\
		\hline
		\hspace{-0.1cm}$T'_{\mr{d0}i}$ & \hspace{-0.1cm} direct-axis open-circuit time constant ($\mr{sec}$) \\
		\hline
		\hspace{-0.1cm}$T_{\mr{CH}i}$& \hspace{-0.1cm} chest valve time constant ($\mr{sec}$)\\
		\hline
		\hspace{-0.1cm}$R_{\mr{D}i}$& \hspace{-0.1cm} speed governor regulation constant  ($\mathrm{Hz/pu}$) \\
		\hline
		\hspace{-0.1cm}$\omega_{0}$& \hspace{-0.1cm} synchronous speed ($2\pi60\;\mathrm{rad/sec}$) \\
		\hline
		\hspace{-0.1cm}$P_{\mr{G}i},\;Q_{\mr{G}i}$ & \hspace{-0.1cm} generator's active and reactive power ($\mr{pu}$)  \\
		\hline
		\hspace{-0.1cm}$P_{\mr{R}i},\;Q_{\mr{R}i}$ & \hspace{-0.1cm} renewable's active and reactive power ($\mr{pu}$)  \\
		\hline
		\hspace{-0.1cm}$P_{\mr{L}i},\;Q_{\mr{L}i}$ & \hspace{-0.1cm} load's active and reactive power ($\mr{pu}$)  \\
		\hline
		\hspace{-0.1cm}$\bar{v}_i = {v}_ie^{j\theta_i}$ & \hspace{-0.1cm} complex bus voltage ($\mr{pu}$)  \\
		\hline
		\hspace{-0.1cm}$\m x_d\in \mbb{R}^{n_d}$ & \hspace{-0.1cm} dynamic states \\ %
		\hline
		\hspace{-0.1cm}$\m x_a\in \mbb{R}^{n_a}$ & \hspace{-0.1cm} algebraic states \\ 
		\hline
		\hspace{-0.1cm}$\m u\in \mbb{R}^{n_u}$ & \hspace{-0.1cm} system's overall inputs \\ 
		\hline
		\hspace{-0.1cm}$\m q\in \mbb{R}^{n_q}$ & \hspace{-0.1cm} demand and renewables generation \\ 
		\hline
	\end{tabular}
	\vspace{-0.5cm}
\end{table}

	\vspace{-0.0cm}
\section{Description of Power Network Dynamics}\label{sec:modeling}
We consider a power network consisting $N$ number of buses, modeled by a graph $(\mathcal{N},\mathcal{E})$ where $\mathcal{N}$ is the set of nodes and $\mathcal{E}$ is the set of edges. 
{Note that $\mathcal{N}$ consists of traditional synchronous generator, renewable energy resources, and load buses, i.e.,  $\mathcal{N} = \mathcal{G} \cup \mathcal{R} \cup \mathcal{L} \cup \mathcal{U}$ where $\mathcal{G}$ collects $G$ generator buses, $\mathcal{R}$ collects the buses containing $R$ renewables, $\mathcal{L}$ collects $L$ load buses, and $\mathcal{U}$ collects $U$ non-unit buses}---see Tab. \ref{tab:notation} for a description of notations. 
 In this paper, we consider a \nth{4}-order dynamics of  synchronous generators modeled as~\cite{sauer2017power,Taha2019TCNS}
\begin{subequations} \label{eq:SynGen}
	\begin{align}
	\dot{\delta}_{i} &= \omega_{i} - \omega_{0} \label{eq:SynGen1} \\ 
	\begin{split}
	M_{i}\dot{\omega}_{i} &= T_{\mr{M}i}-P_{\mr{G}i}- D_{i}(\omega_{i}-\omega_{0}) \end{split}\label{eq:SynGen2}    \\ 
	T'_{\mr{d0}i}\dot{E}'_{i} &= -\tfrac{x_{\mr{d}i}}{x'_{\mr{d}i}}E'_{i} +\tfrac{x_{\mr{d}i}-x'_{\mr{d}i}}{x'_{\mr{d}i}}v_i\cos(\delta_{i}-\theta_i) + E_{\mr{fd}i}  \label{eq:SynGen3} \\
	T_{\mr{CH}i}\dot{T}_{Mi} &= -T_{\mr{M}i} - \tfrac{1}{R_{\mr{D}i}}(\omega_{i}-\omega_{0}) + T_{\mr{r}i}. \label{eq:SynGen4}  
	\end{align} 
\end{subequations}
The time-varying components in \eqref{eq:SynGen} include: generator's internal states $\delta_{i}$, $\omega_{i}$, $E'_{i}$, $T_{\mr{M}i}$; generator's inputs $E_{\mr{fd}i}$, $T_{\mr{r}i}$. 
The relations among generator's internal states $(\delta_{i},\omega_{i},E'_{i},T_{\mr{M}i})$, generator's supplied power $(P_{\mr{G}i},Q_{\mr{G}i})$, and terminal voltage $\bar{v}_i$ are represented by
 two algebraic constraints below~\cite{Taha2019TCNS}
\begin{subequations}\label{eq:SynGenPower}
	\begin{align}
		\begin{split}
		\hspace{-0.3cm}P_{\mr{G}i} &= \tfrac{1}{x'_{\mr{d}i}}E'_{i}v_i\sin(\delta_i-\theta_i) -\tfrac{x_{\mr{q}i}-x'_{\mr{d}i}}{2x'_{\mr{d}i}x_{\mr{q}i}}v_i^2\sin(2(\delta_i-\theta_i))
		\end{split}
	 \label{eq:SynGenPower1} \\
	\begin{split}
		\hspace{-0.3cm}Q_{\mr{G}i} &= \tfrac{1}{x'_{\mr{d}i}}E'_{i}v_i\cos(\delta_i-\theta_i)-\tfrac{x'_{\mr{d}i}+x_{\mr{q}i}}{2x'_{\mr{d}i}x_{\mr{q}i}}v_i^2\\
		\hspace{-0.3cm}&\quad -\tfrac{x_{\mr{q}i}-x'_{\mr{d}i}}{2x'_{\mr{d}i}x_{\mr{q}i}}v_i^2\cos(2(\delta_i-\theta_i)).
	\end{split}\label{eq:SynGePower2}
	\end{align}
\end{subequations}
The power flow/balance equations---which resemble the power transfer among generators, renewable energy resources, and loads---are given as follows~\cite{sauer2017power}  
\begingroup
\allowdisplaybreaks 
\begin{subequations} \label{eq:GPF}
	\begin{align} 
	\begin{split}
\hspace{-0.4cm}P_{\mr{G}i} + P_{\mr{R}i}	+P_{\mr{L}i} \hspace{-0.05cm}&=\hspace{-0.05cm} \sum_{j=1}^{N}\hspace{-0.05cm} v_iv_j\hspace{-0.05cm}\left(G_{ij}\cos \theta_{ij} \hspace{-0.05cm}+ \hspace{-0.05cm}B_{ij}\sin \theta_{ij}\right)
\end{split}\label{eq:GPF1}\\
	\begin{split}
	\hspace{-0.4cm}Q_{\mr{G}i} + Q_{\mr{R}i}	+Q_{\mr{L}i}\hspace{-0.05cm} &=\hspace{-0.05cm} \sum_{j=1}^{N}\hspace{-0.05cm} v_iv_j\hspace{-0.05cm}\left(G_{ij}\sin \theta_{ij} \hspace{-0.05cm}- \hspace{-0.05cm}B_{ij}\cos \theta_{ij}\right),
\end{split}\label{eq:GPF2}		
	\end{align}
\end{subequations}
\endgroup
where $i \in \mathcal{G} \cap \mathcal{R} \cap \mathcal{L}$, $\theta_{ij}:= \theta_i-\theta_j$, and $(G_{ij},B_{ij})$ respectively denote the conductance and susceptance between bus $i$ and $j$ which can be directly obtained from the network's bus admittance matrix \cite{sauer2017power}. In the above equations, $(P_{\mr{R}i},Q_{\mr{R}i})$ denote the active and reactive power generated by renewables, while $(P_{\mr{L}i},Q_{\mr{L}i})$ denote the active and reactive power consumed by the loads. For the case at which a bus does not contain generator, renewable, and/or load, then the absence of one or more of these units can be indicated by setting its/their corresponding active and reactive power in \eqref{eq:GPF} to zero.
Now, let us define:
${\m x}_d$ as the vector populating all dynamic states of the network such that
${\m x}_d := \bmat{\m \delta^\top\;\;\m \omega^\top\;\;\m E'^\top\;\;\m T_{\mr{M}}^\top}^\top$ in which ${\m \delta}\hspace{-0.05cm}:=\hspace{-0.05cm}\{\delta_i\}_{i\in \mc{G}}\hspace{-0.05cm}$, ${\m \omega}\hspace{-0.05cm}:=\hspace{-0.05cm}\{\omega_i\}_{i\in \mc{G}}\hspace{-0.05cm}$,  ${\m E'}\hspace{-0.05cm}:=\hspace{-0.05cm}\{E'_i\}_{i\in \mc{G}}\hspace{-0.05cm}$, ${\m T_{\mr{M}}}\hspace{-0.05cm}:=\hspace{-0.05cm}\{T_{\mr{M}i}\}_{i\in \mc{G}}\hspace{-0.05cm}$; ${\m a}$ as the algebraic state corresponding to generator's power such that ${\m a} :=  \bmat{\m P_{\mr{G}}^{\top} \;\; \m Q_{\mr{G}}^{\top}}^{\top}$ where ${\m P_{\mr{G}}}\hspace{-0.05cm}:=\hspace{-0.05cm}\{P_{\mr{G}i}\}_{i\in \mc{G}}\hspace{-0.05cm}$, ${\m Q_G}\hspace{-0.05cm}:=\hspace{-0.05cm}\{Q_{\mr{G}i}\}_{i\in \mc{G}}\hspace{-0.05cm}$; and $\tilde{\m v}$ as the algebraic state representing the network's complex bus voltages such that $\tilde{\m v} :=  \bmat{\m v^\top \;\; \m \theta^\top}^{\top}$ where ${\m v}\hspace{-0.05cm}:=\hspace{-0.05cm}\{v_i\}_{i\in \mc{N}}\hspace{-0.05cm}$, ${\m \theta}\hspace{-0.05cm}:=\hspace{-0.05cm}\{\theta_i\}_{i\in \mc{N}}\hspace{-0.05cm}$. The input of the system is considered to be ${\m u} := \bmat{\m E_{\mr{fd}}^\top\;\;\m T_{\mr{r}}^\top}^\top$ where  
${\m E_{\mr{fd}}}\hspace{-0.05cm}:=\hspace{-0.05cm}\{E_{\mr{fd}i}\}_{i\in \mc{G}}\hspace{-0.05cm}$ and ${\m T_{\mr{r}}}\hspace{-0.05cm}:=\hspace{-0.05cm}\{T_{\mr{r}i}\}_{i\in \mc{G}}\hspace{-0.05cm}$. In addition, define the vector ${\m q}$ as ${\m q} :=  \bmat{\m P_{\mr{R}}^{\top} \;\; \m Q_{\mr{R}}^{\top}\;\; \m P_{\mr{L}}^{\top} \;\; \m Q_{\mr{L}}^{\top}}^{\top}$ where ${\m P_{\mr{R}}}\hspace{-0.05cm}:=\hspace{-0.05cm}\{P_{\mr{R}i}\}_{i\in \mc{R}}\hspace{-0.05cm}$, ${\m Q_{\mr{R}}}\hspace{-0.05cm}:=\hspace{-0.05cm}\{Q_{\mr{R}i}\}_{i\in \mc{R}}\hspace{-0.05cm}$, ${\m P_{\mr{L}}}\hspace{-0.05cm}:=\hspace{-0.05cm}\{P_{\mr{L}i}\}_{i\in \mc{L}}\hspace{-0.05cm}$, ${\m Q_{\mr{L}}}\hspace{-0.05cm}:=\hspace{-0.05cm}\{Q_{\mr{L}i}\}_{i\in \mc{L}}\hspace{-0.05cm}$.
Based on the constructed vectors described above, the state-space, NDAE model of multi-machine power networks \eqref{eq:SynGen}-\eqref{eq:GPF} can be written as
\begin{subequations}\label{eq:nonlinearDAEexplicit}
		\begin{empheq}[box=\widefbox]{align}
			\m E_d\dot{{\m x}}_d &= {\m A}_d{\m x}_d +  {\m G}_d{\m f}_d\left({\m x}_d,{\m x}_a\right) + {\m B}_d {\m u} + {\m h} \omega_{0}\label{eq:nonlinearDAEexplicit-1}\\
		\m 0 &= {\m A}_a{\m x}_a + {\m G}_a{\m f}_a\left({\m x}_d,{\m x}_a\right) + {\m B}_a {\m q}\label{eq:nonlinearDAEexplicit-2},
	\end{empheq}
\end{subequations}
where $\m x_d \in \mbb{R}^{n_d}$, $\m x_a := \bmat{{\m a}^\top\;\;\tilde{\m v}^\top}^\top\in \mbb{R}^{n_a}$, $\m u \in \mbb{R}^{n_u}$, and $\m q \in \mbb{R}^{n_q}$. The functions $\m f_d:\mathbb{R}^{n_d}\times \mathbb{R}^{n_a}\rightarrow \mathbb{R}^{n_{fd}}$, $\m f_a:\mathbb{R}^{n_d}\times \mathbb{R}^{n_a}\rightarrow \mathbb{R}^{n_{fa}}$, constant matrices ${\m A}_d\in \mbb{R}^{n_d\times n_d}$, ${\m A}_a\in \mbb{R}^{n_a\times n_a}$, ${\m G}_d\in \mbb{R}^{n_{fd}\times n_{d}}$, ${\m G}_a\in \mbb{R}^{n_{fa}\times n_{a}}$, ${\m B}_d\in \mbb{R}^{n_u\times n_{d}}$, ${\m B}_a\in \mbb{R}^{n_q\times n_{a}}$, and vector $\m h\in \mbb{R}^{n_d}$ are all detailed in Appendix \ref{appdx:A}. {In \eqref{eq:nonlinearDAEexplicit}, we have $\m E_d = \m I$ for this model\footnote{The matrix $\m E_d$ is kept in the controller derivations for the sake of generality since the state-space representation of power networks \eqref{eq:SynGen}-\eqref{eq:GPF} is \textit{not} unique and thus, it is possible to have $\m E_d \neq \m I$.}.}
The ensuing sections describe the development of a LRFC law $\m u(t)$ for power networks modeled in \eqref{eq:nonlinearDAEexplicit} is presented.

\begin{figure}
	\vspace{-0.05cm}
	\centering 
	{\includegraphics[keepaspectratio=true,scale=0.6]{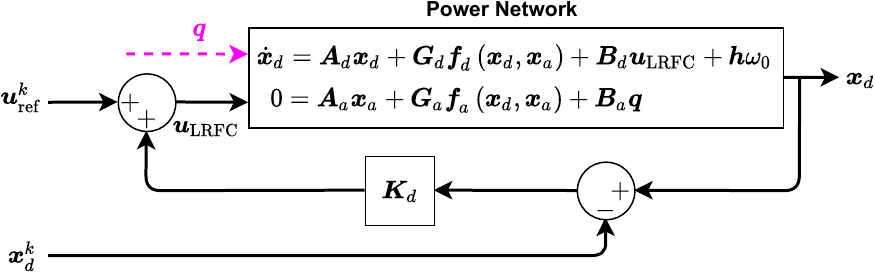}}
	\vspace{-0.6cm}
	\caption{Control architecture for LRFC. Vector $\m q$ denotes the actual demand and renewables generation which values are generally unknown.}
	\label{fig:control_architecture}\vspace{-0.5cm}
\end{figure}

\vspace{-0.0cm}
\section{State Feedback Control Design for NDAEs }\label{sec:control_design}

\subsection{State Feedback Control Strategy for LRFC}\label{ssec:feedback_structure}
The scheduling of synchronous generators in power networks is performed based on the loads and renewables demand and production forecasts. These day-ahead forecasts provide hourly figures of power demand and production \cite{hodge2012comparison}. {Based on this data and assuming that the power system operates in a quasi steady-state, the independent system operator solves the power flow (PF) or optimal power flow (OPF) given in \eqref{eq:GPF} every $T$ minutes---typical value is $15$ minutes or so---to aid the primary, secondary and tertiary controls \cite{Faulwasser2018}.} Each solution obtained from solving the PF/OPF corresponds to a particular operating point (also known as \textit{equilibrium}). To describe how the proposed LRFC is implemented, consider an ideal case when the actual demand and power production by the renewables, denoted by $\m q(t)$, are known and static over a short time period $kT$ where $k \geq 0$ indicates the discrete-time index---let $\m q^k$ be the predicted demand and renewable generation such that $\m q(t) = \m q^k$ where $kT \leq t \leq (k+1)T$. As such, the system rests at equilibrium with $(\m x_d^k,\m x_a^k)$ denoting the steady-state dynamic and algebraic states while $\m u_{\mr{ref}}^k$ denoting the steady-state generators' inputs. 

Since the power supply and demand are balanced, then we have $\omega_{i} = \omega_{0}$ for all $i\in\mc{G}$. Yet, in reality, the values of $\m q(t)$ are highly stochastic and rapidly changing over time. In order to maintain the system's frequency as close to $60\;\mr{Hz}$ as possible, when $\m q(t) \neq \m q^k$ due to demand and renewables variability, the new $\m u_{\mr{ref}}^{k}$ has to be computed and this must be followed by solving the PF/OPF. This practice is impractical since $\m q(t) \neq \m q^k$ might happen during $kT \leq t \leq (k+1)T$ and especially when the deviations are relatively small. As a means to sustain the system's frequency at $60\;\mr{Hz}$ while still being able to solve the PF/OPF within the $15$ minutes interval, we propose a state feedback control architecture in which the controller gain matrix is independent of the solution of the PF/OPF. The power network's dynamics with such a controller are written as
\begin{figure}
	\vspace{-0.1cm}
	\centering 
	{\includegraphics[keepaspectratio=true,scale=0.73]{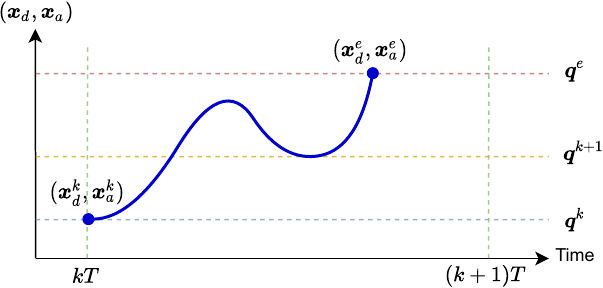}}
	\vspace{-0.3cm}
	\hspace{-1cm}
	\caption{The LRFC is intended to stabilize the system---the actual state during $kT \leq t \leq (k+1)T$ is represented by the blue curve---when the actual power demand and renewable generation are transiting from the projected $\m q^k$ at time $kT$ to the new level $\m q^e$ until $(k+1)T$. }
	\label{fig:illustration}
	\vspace{-0.6cm}
\end{figure}
\begin{subequations}\label{eq:nonlinearDAEexplicitControl}
	\begin{align}
\m E_d\dot{{\m x}}_d &= {\m A}_d{\m x}_d +  {\m G}_d{\m f}_d\left({\m x}_d,{\m x}_a\right)  + {\m B}_d{\m u}_{\mathrm{LRFC}} + {\m h} \omega_{0}
		\label{eq:nonlinearDAEexplicitControl-1}\\
 \m 0 &= {\m A}_a{\m x}_a + {\m G}_a{\m f}_a\left({\m x}_d,{\m x}_a\right) + {\m B}_a {\m q}\label{eq:nonlinearDAEexplicitControl-2},
	\end{align}
\end{subequations} 
where the control input during $kT \leq t \leq (k+1)T$ is given as
$$\boxed{{\m u}_{\mathrm{LRFC}}:= \m u_{\mathrm{LRFC}}(t) = \m u_{\mathrm{ref}}^k + \m K_d \left({\m x}_d(t) - \m x_d^k\right),}$$
in which $\m K_d\in\mbb{R}^{n_u \times n_d}$ denotes the associated controller gain matrix. In this approach, $\m K_d$ is computed based only on the knowledge of matrices and functions provided in \eqref{eq:nonlinearDAEexplicitControl} and thus {independent} from $\m u_{\mathrm{ref}}^k$ and $(\m x_d^k,\m x_a^k)$. The overall structure of the proposed LRFC is depicted in Fig. \ref{fig:control_architecture}. This control architecture only \textit{(i)} requires the knowledge of generators' internal states---thus does not rely on any real-time measurements of algebraic variables $(\m a,\tilde{\m v})$ whatsoever---while \textit{(ii)} not involving any kind of system's linearization around $(\m x_d^k,\m x_a^k)$. 
It is worth noting that the control architecture depicted in Fig. \ref{fig:control_architecture} is common in power systems secondary control \cite{siljak2002robust,Marinovici2013,Taha2019TCNS}.
Now, suppose that a disturbance---attributed to a sudden change in power demands and/or power produced by the renewables---is applied to the network. This disturbance will eventually throw the system's operating point to a new equilibrium. Let us denote $\m q^e$ as the new actual demand and generated power from the renewables at $kT \leq t \leq (k+1)T$ time instance. Using the proposed LRFC framework described in \eqref{eq:nonlinearDAEexplicitControl}, the system's dynamics at the new steady-state operating point indicated by $({\m x}_d^e,{\m x}_a^e)$ can be expressed as
 \begin{subequations}\label{eq:nonlinearDAEpostDist}
	\begin{align}
		\begin{split}
			 \m 0 &= {\m A}_d{\m x}_d^e +  {\m G}_d{\m f}_d\left({\m x}_d^e,{\m x}_a^e\right)  + {\m h} \omega_{0} \\
			&\quad + {\m B}_d\left( \m u_{\mathrm{ref}}^k + \m K_d \left({\m x}_d^e - \m x_d^k\right)\right)
		\end{split}\label{eq:nonlinearDAEpostDist-1}\\
		\m 0 &= {\m A}_a{\m x}_a^e + {\m G}_a{\m f}_a\left({\m x}_d^e,{\m x}_a^e\right) + {\m B}_a {\m q^e}\label{eq:nonlinearDAEpostDist-2}.
	\end{align}
\end{subequations} 
In order to analyze the network's dynamical behavior after the disturbance is initiated, let us introduce $\Delta \m x_d\in\mbb{R}^{n_d}$ and $\Delta \m x_a\in\mbb{R}^{n_a}$ as the \textit{deviations} of the dynamic and algebraic states of the perturbed system around $({\m x}_d^e,{\m x}_a^e)$, respectively, and they are given as $\Delta \m x_d := \m x_d-\m x_d^e$ and $\Delta \m x_a := \m x_a-\m x_a^e$. From \eqref{eq:nonlinearDAEexplicitControl}, \eqref{eq:nonlinearDAEpostDist}, and letting $\Delta\m q :=\m q - {\m q^e}$, the perturbed network's dynamics can be derived as
\begin{subequations}\label{eq:nonlinearDAEpert}
\begin{empheq}[box=\widefbox]{align}
		\m E_d\Delta\dot{{\m x}}_d \hspace{-0.05cm}&=\hspace{-0.05cm}({\m A}_d\hspace{-0.05cm}+\hspace{-0.05cm}{\m B}_d \m K_d)\Delta{\m x}_d \hspace{-0.05cm}+\hspace{-0.05cm}{\m G}_d\Delta \m f_d\left(\Delta{\m x}_d,\Delta{\m x}_a\right) \label{eq:nonlinearDAEpert-1}\\
		 \m 0\hspace{-0.05cm} &= \hspace{-0.05cm}{\m A}_a\Delta {\m x}_a\hspace{-0.05cm}+\hspace{-0.05cm}{\m G}_a\Delta \m f_a\left(\Delta{\m x}_d,\Delta{\m x}_a\right)\hspace{-0.05cm}+\hspace{-0.05cm}{\m B}_a \Delta\m q \label{eq:nonlinearDAEpert-2},
	\end{empheq}
\end{subequations}  
\noindent where the mappings $\Delta \m f_d(\cdot)$ and $\Delta \m f_a(\cdot)$ are detailed as {
\begin{align*}
\Delta \m f_d\left({\m x},{\m x}^e\right) &:= {\m f}_d\left({\m x}_d,{\m x}_a\right) - {\m f}_d\left({\m x}_d^e,{\m x}_a^e\right) \\
\Delta \m f_a\left({\m x},{\m x}^e\right) &:= {\m f}_a\left({\m x}_d,{\m x}_a\right) - {\m f}_a\left({\m x}_d^e,{\m x}_a^e\right),
\end{align*}
where $\m x := \bmat{\m x_d^\top\,\,\m x_a^\top}^\top$ (likewise for $\m x^e$).}
In \eqref{eq:nonlinearDAEpert}, $\Delta\m q$ reflects the deviation of the current demand and renewables generation $\m q$ from the new operating values $\m q^e$ and as such, $\Delta\m q$ is considered to be relatively small ($\Delta\m q\approx \m 0$). 
Our objective herein is to design/compute $\m K_d$ such that all trajectories of the solutions of the NDAE \eqref{eq:nonlinearDAEpert} will converge asymptotically towards the zero equilibrium. This is equivalent for the states of power network \eqref{eq:nonlinearDAEpostDist} to converge towards the new operating point indicated by $({\m x}_d^e,{\m x}_a^e)$. This process is illustrated in Fig. \ref{fig:illustration}.

\vspace{-0.1cm}
\subsection{Stabilization of Power Network's NDAEs}\label{ssec:control_design}
To simplify the notations, let $\check{\m x}_d := \Delta \m x_d$, $\check{\m x}_a := \Delta \m x_a$, $\check{\m f}_d := \Delta \m f_d$, and $\check{\m f}_a := \Delta \m f_a$ such that \eqref{eq:nonlinearDAEpert} can be written as 
\begin{subequations}\label{eq:nonlinearDAEpertGen}
	\begin{align}
		\m E_d\dot{\check{\m x}}_d &= ({\m A}_d +  {\m B}_d \m K_d)\check{\m x}_d +  {\m G}_d\check{\m f}_d\left({\m x},{\m x}^e\right) \label{eq:nonlinearDAEpertGen-1}\\
	\m 0 &= {\m A}_a \check{\m x}_a + {\m G}_a \check{\m f}_a\left({\m x},{\m x}^e\right)\label{eq:nonlinearDAEpertGen-2}.
	\end{align}
\end{subequations}
 Albeit the NDAE \eqref{eq:nonlinearDAEpertGen} assumes that $\Delta\m q = \m 0$, in Section \ref{sec:numerical_studies} we study the performance of the LRFC when disturbances are present and therefore, the stability of \eqref{eq:nonlinearDAEpert} is studied against a nonzero disturbance.
 It is also assumed herein that $\check{\m x}_d\in \mathbfcal{X}_d \subseteq \mathbb{R}^{n_d}$ and $\check{\m x}_a\in \mathbfcal{X}_a \subseteq \mathbb{R}^{n_a}$. That is, the sets $\mathbfcal{X}_d$ and $\mathbfcal{X}_a$ represent the operating region(s) of the power networks and contain the solution manifold of \eqref{eq:nonlinearDAEpertGen}. The following assumptions (which are standard in the literature on control and stabilization of DAEs \cite{Franco2019,Franco2020}) are crucial for the development of our LRFC method and therefore considered to hold throughout the paper.
\vspace{-0.1cm}
\begin{asmp}\label{asmp:nonlinearity}
	The following properties hold for the mappings $\check{\m f}_d:\mathbb{R}^{n_d}\times \mathbb{R}^{n_a}\rightarrow \mathbb{R}^{n_{fd}}$ and $\check{\m f}_a:\mathbb{R}^{n_d}\times \mathbb{R}^{n_a}\rightarrow \mathbb{R}^{n_{fa}}$:
	\begin{enumerate}[leftmargin=*]
		\item $\check{\m f}_d(\cdot)$ and $\check{\m f}_a(\cdot)$ are smooth and satisfy $\check{\m f}_d(\m 0,\m 0) = \m 0$ and $\check{\m f}_a(\m 0,\m 0) = \m 0$.
		\item $\check{\m f}_d(\cdot)$ and $\check{\m f}_a(\cdot)$ are quadratically-bounded functions such that, given $\check{\m x}_d \in \mathbfcal{X}_d$ and $\check{\m x}_a \in \mathbfcal{X}_a$, it holds that{
		\begin{subequations}\label{eq:qb_condition}
			\begin{align}
				\hspace{-0.4cm}\norm{\check{\m f}_d\left({\m x}(t),{\m x}^e(t)\right)}_2^2 &\leq \norm{\m H_d^d \check{\m x}_d(t)}_2^2 + \norm{\m H_a^d \check{\m x}_a(t)}_2^2 \label{eq:qb_condition-1}\\
				\hspace{-0.4cm}\norm{\check{\m f}_a\left({\m x}(t),{\m x}^e(t)\right)}_2^2 &\leq \norm{\m H_d^a \check{\m x}_d(t)}_2^2 + \norm{\m H_a^a \check{\m x}_a(t)}_2^2,\label{eq:qb_condition-2}
			\end{align}
		\end{subequations}}
		for some known constant matrices $\m H_d^d,\,\m H_a^d,\,\m H_d^a,\,\m H_a^a$. 
	\end{enumerate}	
\end{asmp}
\vspace{-0.3cm}
\begin{asmp}\label{asmp:index_one}
	This rank equality
	\begin{align}
		\mathrm{rank}\left(\m A_a + {\m G}_a\frac{\partial\check{\m f}_a \left({\m x},{\m x}^e\right)}{\partial \check{\m x}_a}\right) = n_a, \label{eq:index_one}
	\end{align}
	is satisfied for all $\check{\m x}_d \in \mathbfcal{X}_d$ and $\check{\m x}_a \in \mathbfcal{X}_a.$ 
\end{asmp}
\vspace{-0.1cm} 
It is worth mentioning that Assumption \ref{asmp:nonlinearity} is mild in power networks---see \cite{siljak2002robust,Marinovici2013}. In fact, it is shown in \cite{siljak2002robust} that, for a simplified ODE representation of power networks with turbine governor dynamics, there exist bounding matrices such that \eqref{eq:qb_condition-1} holds without the presence of $\check{\m x}_a$.  
In principle, the nonlinearities in the NDAE model are treated as external disturbances originating from the network's interconnections, and as such, since their influence on the system is bounded according to \eqref{eq:qb_condition}, the designed stabilizing controller attempts to compensate for impacts caused by these disturbances.   

In the classical DAE systems theory, the differentiation index can be associated with the minimum number of steps required for expressing the corresponding DAE in an explicit form \cite{Gross2014,duan2010analysis}.
The condition \eqref{eq:index_one} is useful to ensure that the NDAE \eqref{eq:nonlinearDAEpertGen} is of index one \cite{Franco2020}. For a simplified model of multi-machine power networks, it is proved in \cite{Gross2014} that power networks' DAEs are of index one if every load bus is connected to at least one generator bus. Since this is the case in normal conditions (e.g., no tripping in power lines), then Assumption \ref{asmp:index_one} is easily satisfied. Although the property introduced in \cite{Gross2014} is studied for a simplified model without involving any renewables, it is revealed that the condition \eqref{eq:index_one} actually holds for a more comprehensive model of power networks considered in this paper---this is evident from being able to numerically simulate power networks for various test cases (see Section \ref{sec:numerical_studies}). Hence, based on the above assumptions, we now focus on providing a computational approach to calculate the state feedback gain matrix $\m K_d$ such that the NDAE \eqref{eq:nonlinearDAEpertGen} is asymptotically stable. That is, the NDAE \eqref{eq:nonlinearDAEpertGen} is said to be asymptotically stable if $\lim_{t\rightarrow\infty} \norm{\check{\m x}_d(t)}_2 = 0$ and $\lim_{t\rightarrow\infty} \norm{\check{\m x}_a(t)}_2 = 0$ \cite{BoMen2006}.
The following result provides a sufficient condition for the asymptotic stability of NDAE \eqref{eq:nonlinearDAEpertGen} at the origin.  
\vspace{-0.1cm}
\begin{theorem}\label{thm:exp_stabilization}
Consider the NDAE \eqref{eq:nonlinearDAEpertGen} provided that Assumptions \ref{asmp:nonlinearity} and \ref{asmp:index_one} hold. The closed-loop system is asymptotically stable around the origin if there exist matrices $\m Q_1\in\mbb{R}^{n_d\times n_d}$, $\m Q_2\in\mbb{R}^{n_a\times n_d}$, $\m Q_3\in\mbb{R}^{n_a\times n_a}$, where both $\m Q_1$ and $\m Q_3$ are nonsingular, and a scalar $\bar{\epsilon} \in\mbb{R}_{++}$ such that the following matrix inequalities are feasible 
	\begin{subequations}\label{eq:LMI_stabilization_all}
	\begin{align}
		&\hspace{-0.1cm}\bmat{\m \Upsilon & * & * & * \\ \m A_a \m Q_2 & \m Q_3^\top \m A_a^\top + \m A_a\m Q_3+ \bar{\epsilon}\m G_a\m G_a^\top  &* &* \\ \bar{\m H}_d^{\frac{1}{2}}\m Q_1 & \m O & -\bar{\epsilon}\m I & * \\ \bar{\m H}_a^{\frac{1}{2}}\m Q_2 & \bar{\m H}_d^{\frac{1}{2}}\m Q_3 & \m O & -\bar{\epsilon}\m I} \prec 0 \label{eq:LMI_stabilization} \\
		&\hspace{-0.1cm}\qquad  \qquad  \qquad  \qquad  \qquad  \quad \m E_d^\top \m Q_1^{-1} =\m Q_1^{-\top}\m E_d \succ 0, \label{eq:LMI_stabilization_Q}
	\end{align}
	\end{subequations}
where $\m \Upsilon$ includes the matrix $\m K_d$ and is defined as $$ \m Q_1^\top \m A_d^\top + \m A_d\m Q_1 + \m Q_1^\top \m K_d^\top\m B_d^\top \\
+ \m B_d \m K_d \m Q_1 + \bar{\epsilon}\m G_d\m G_d^\top,$$ 
The matrices $\bar{\m H}_d$ and $\bar{\m H}_a$  in \eqref{eq:LMI_stabilization} are specified as 
	\begin{align*}
		\bar{\m H}_d := {\m H}_d^{d\top}\hspace{-0.05cm}{\m H}_d^d + {\m H}_d^{a\top}\hspace{-0.05cm}{\m H}_d^a,\;\;\bar{\m H}_a := {\m H}_a^{d\top}\hspace{-0.05cm}{\m H}_a^d + {\m H}_a^{a\top}\hspace{-0.05cm}{\m H}_a^a.
	\end{align*}
\end{theorem}
\vspace{-0.1cm}
The complete proof of Theorem \ref{thm:exp_stabilization} is available in Appendix \ref{appdx:B}. 
{The feasibility of matrix inequalities \eqref{eq:LMI_stabilization_all} guarantees the existence of $\m K_d$ that asymptotically stabilizes the NDAE \eqref{eq:nonlinearDAEpertGen} around the zero equilibrium.} Realize that, since the states of the NDAE \eqref{eq:nonlinearDAEpertGen} in fact are just the deviations of the actual states $(\m x_d,\m x_a)$ from the new operating point $(\m x_d^e,\m x_a^e)$, it can be easily deduced that 
\begin{align*}
	\lim_{t\rightarrow \infty} \Delta \m x_d(t) &= 0 \;\;\Rightarrow \;\; \lim_{t\rightarrow \infty} \m x_d(t) = \m x_d^e \\ 
	\lim_{t\rightarrow \infty} \Delta \m x_a(t) &= 0 \;\;\Rightarrow \;\; \lim_{t\rightarrow \infty} \m x_a(t) = \m x_a^e. 
\end{align*}
Since $\m x_d^e$ consists of the synchronous frequency for the rotors of all rotating machines, then we have $\omega_{i} = \omega_{0}$ for all $i\in\mc{G}$. In short, the proposed state feedback control strategy with gain matrix $\m K_d$ is able to provide LRFC due to the changes in power demands and renewables generation.
Unfortunately, the majority of off-the-shelf optimization packages, e.g. YALMIP \cite{Lofberg2004}, cannot be utilized to find solutions for \eqref{eq:LMI_stabilization_all} due to the nonconvexity of the problem, which is partly attributed to the appearance of $\m Q_1^{-1}$ in \eqref{eq:LMI_stabilization_Q} along with the existence of bilinear term $\m K_d \m Q_1$.
To circumvent this design challenge, the following result is proposed. 
\vspace{-0.1cm}
\begin{myprs}\label{prs:LMI_stability}
	Consider the NDAE \eqref{eq:nonlinearDAEpertGen} given that Assumptions \ref{asmp:nonlinearity} and \ref{asmp:index_one} hold. The closed-loop system is asymptotically stable around the origin if there are matrices $\m X_1\in\mbb{S}^{n_d}_{++}$, $\m X_2\in\mbb{R}^{n_a\times n_d}$, $\m R\in\mbb{R}^{n_a\times n_a}$, $\m Y\in\mbb{R}^{n_d\times n_a}$, $\m W\in\mbb{R}^{n_u\times n_d}$,  and a scalar $\bar{\epsilon} \in\mbb{R}_{++}$ such that the following LMI is feasible
	\begin{align}
		\bmat{\m \Psi & * & * & * \\ \m A_a \m X_2\m E_d^\top + \m A_a \m Y & \m \Theta  &* &* \\ \bar{\m H}_d^{\frac{1}{2}}\m X_1\m E_d^\top & \m O & -\bar{\epsilon}\m I & * \\ \bar{\m H}_a^{\frac{1}{2}}\m X_2\m E_d^\top + \bar{\m H}_a^{\frac{1}{2}} \m Y & \bar{\m H}_d^{\frac{1}{2}}\m R & \m O & -\bar{\epsilon}\m I} &\prec 0, \label{eq:LMI_stabilization_linear}
	\end{align}
where $\m \Psi$ is specified as $$ \m E_d\m X_1\m A_d^\top + \m A_d\m X_1\m E_d^\top + \m E_d\m W^\top \m B_d^\top 
+ \m B_d \m W \m E_d^\top + \bar{\epsilon}\m G_d\m G_d^\top,$$ and  $\m \Theta := \m R^\top \m A_a^\top + \m A_a\m R+ \bar{\epsilon}\m G_a\m G_a^\top$. Upon solving \eqref{eq:LMI_stabilization_linear}, the controller gain $\m K_d$ can be recovered as $\m K_d = \m W\m X_1^{-1}$.
\end{myprs}
\vspace{-0.1cm}
\noindent Readers are referred to Appendix \ref{appdx:C} for the proof of Proposition \ref{prs:LMI_stability}. In contrast to matrix inequality \eqref{eq:LMI_stabilization_all}, the one given in \eqref{eq:LMI_stabilization_linear} constitutes an LMI and therefore can be easily solved through standard convex optimization packages. 

For some practical reasons, it is often highly desired to obtain small feedback gains so that the resulting transient behaviors can be kept within acceptable bounds and do not strain the system protection \cite{Marinovici2013}. Contrary, a high gain controller is in general undesirable since it could increase the sensitivity of the closed-loop system against noise and uncertainty. To that end, we consider solving the following optimization problem in the interest of obtaining $\m K_d$ with a reasonable magnitude.
\begin{eqnarray*}
	\mathbf{(P)}\;\;\;\minimize_{\bar{\epsilon}, \m X_1, \m X_2, \m R, \m Y, \m W} &   & \norm{\m W}_2 \\
	\subjectto &  &  \eqref{eq:LMI_stabilization_linear} ,\;  \m X_1 \succ 0,\;\bar{\epsilon} > 0,
\end{eqnarray*} 
where $\norm{\m W}_2$ denotes the induced $2$-norm of matrix $\m W$. 

\vspace{0.05cm}
\setlength{\textfloatsep}{5pt}
{\small \begin{algorithm}[t]
		\caption{\text{Implementation of The LRFC}}\label{alg:LRFC}
		\DontPrintSemicolon 
		\textbf{input:} ${\m A}_d$, ${\m A}_a$, ${\m G}_d$, ${\m G}_a$, ${\m B}_d$, ${\m B}_a$, $T$\;
		\textbf{compute:} $\m K_d$ by solving problem $\mathbf{P}$\;
		\textbf{initialize:} iteration index $k = 0$ \;
		\Do{$k < \infty$ \label{alg:stop}}{
			\textbf{obtain:} $\m q^k$ from prediction and measurement \;
			\textbf{solve:} PF/OPF based on $\m q^k$\;
			\textbf{get:} $\m x_a^k$ from the solution of PF/OPF\;
			\textbf{compute:} $\m u_{\mr{ref}}^k$ and $\m x_d^k$ \;
			\textbf{update:} the LRFC in Fig. \ref{fig:control_architecture} with $(\m x_d^k,\m u_{\mr{ref}}^k)$  \;
		\textbf{wait:} $T$ minutes \textcolor{blue}{// \textbf{\textit{OPF Time-Period}}}\;
		\textbf{update:} $k\leftarrow k + 1$\;
		}
	\end{algorithm}
}
\setlength{\floatsep}{5pt}  

\vspace{-0.4cm}
\subsection{Implementation of The Proposed LRFC Strategy}\label{ssec:implementation}
The proposed LRFC strategy can be implemented as follows. First, based on the matrices describing the network dynamics \eqref{eq:nonlinearDAEexplicit}, the controller gain $\m K_d$ is computed by solving problem $\mathbf{P}$. Based on the load and renewable forecasts $\m q^k$, the steady-state algebraic variables $\m x_a^k$ can be obtained by solving the PF/OPF. Afterwards, $(\m x_d^k,\m u_{\mr{ref}}^k)$ can be computed by setting $\dot{\m x}_d = 0$ in \eqref{eq:nonlinearDAEexplicit} and from $(\m q^k,\m x_a^k)$, the resulting system of nonlinear equations is numerically solved. The calculated $(\m x_d^k,\m u_{\mr{ref}}^k)$ is then fed to the control architecture illustrated in Fig. \ref{fig:control_architecture}. These steps are then repeated once every $T$, which is typically around $15$ minutes \cite{Faulwasser2018}, to continuously perform LRFC and compensate for any changes in demand and renewables generation. Algorithm \ref{alg:LRFC} presents a summary of how the LRFC is implemented. Realize that, since the matrix $\m K_d$ is only computed \textit{once}, our approach for LRFC is much more practical compared to other methods that rely on the linearization of \eqref{eq:nonlinearDAEexplicit} around the operating point $(\m x_d^k,\m x_a^k)$ because, in addition to solving the PF/OPF and the set of nonlinear equations mentioned above, the independent system operator has to \textit{(a)} perform the linearization while also \textit{(b)} computing the stabilizing controller gain matrix---two are carried out in each iteration within $kT \leq t \leq (k+1)T$ time interval. This linearization-based approach certainly necessitates more demanding computational processes to be performed.  
\vspace{-0.10cm}
\begin{myrem}\label{rem:1}
Despite the proposed LRFC strategy does not consider impacts caused by parametric uncertainties, one can perform  sensitivity analyses to predict the levels of uncertainty propagation within a certain time period \cite{Choi2017}, after which the predicted worst-case operating regions can be determined and included in the sets $\mathbfcal{X}_d$ and $\mathbfcal{X}_a$.
\end{myrem}
\vspace{-0.20cm}

\section{Numerical Case Studies}\label{sec:numerical_studies} 

\begin{figure*}
	\centering 
	\subfloat[]{\includegraphics[keepaspectratio=true,scale=0.627]{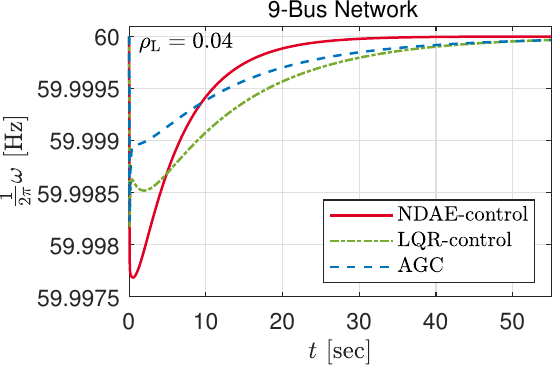}}{}\hspace{-0.08cm}
	\subfloat[]{\includegraphics[keepaspectratio=true,scale=0.627]{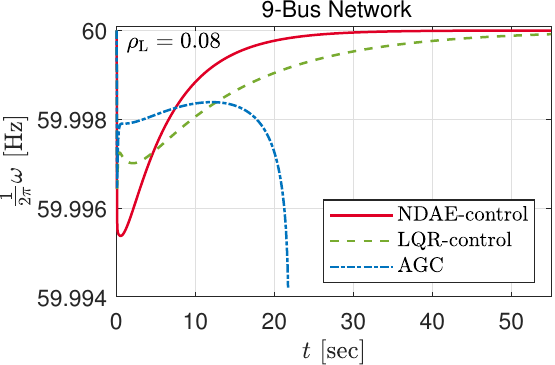}}{}\hspace{-0.08cm}
	\subfloat[]{\includegraphics[keepaspectratio=true,scale=0.627]{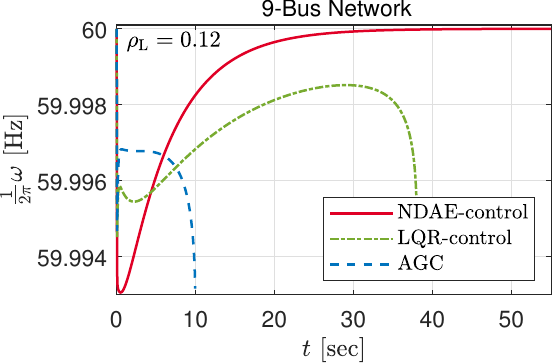}}{}\hspace{-0.08cm}\vspace{-0.25cm}
	\subfloat[]{\includegraphics[keepaspectratio=true,scale=0.627]{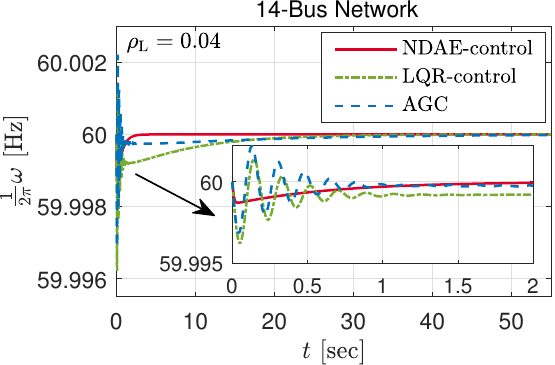}}{}\hspace{-0.08cm}
	\subfloat[]{\includegraphics[keepaspectratio=true,scale=0.627]{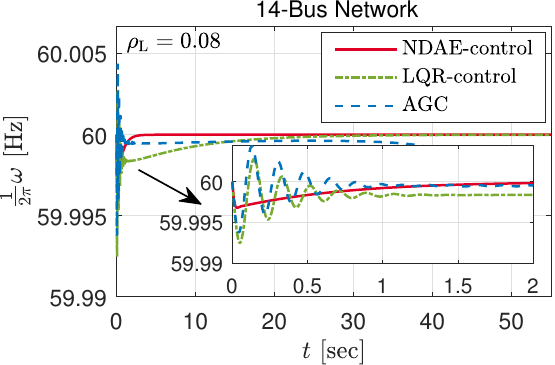}}{}\hspace{-0.08cm}
	\subfloat[]{\includegraphics[keepaspectratio=true,scale=0.627]{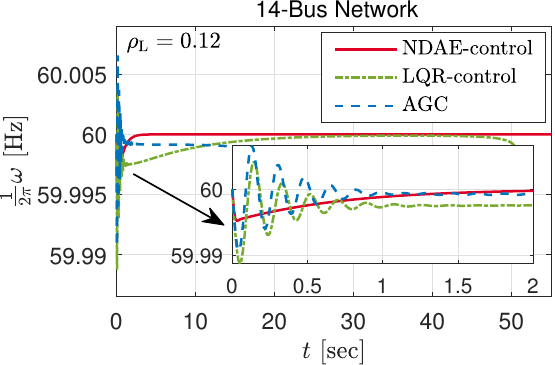}}{}\hspace{-0.08cm}\vspace{-0.25cm}
	\subfloat[]{\includegraphics[keepaspectratio=true,scale=0.627]{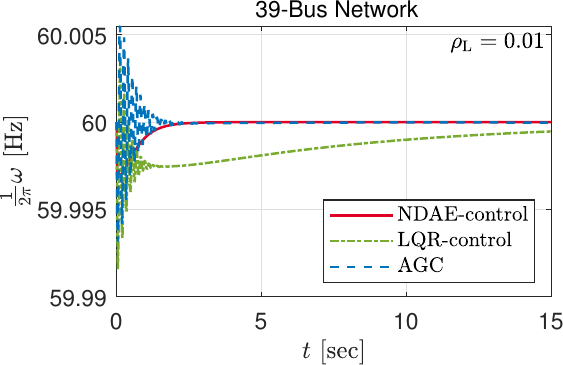}}{}\hspace{-0.08cm}
	\subfloat[]{\includegraphics[keepaspectratio=true,scale=0.627]{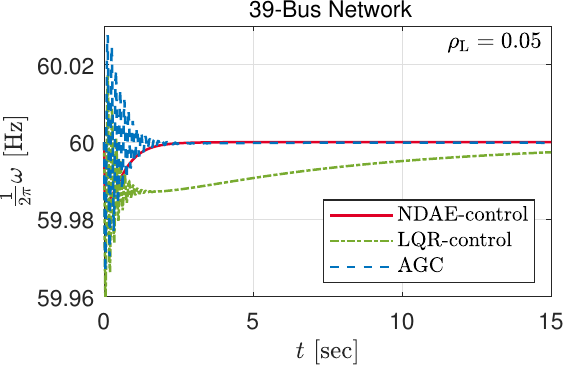}}{}\hspace{-0.08cm}
	\subfloat[]{\includegraphics[keepaspectratio=true,scale=0.627]{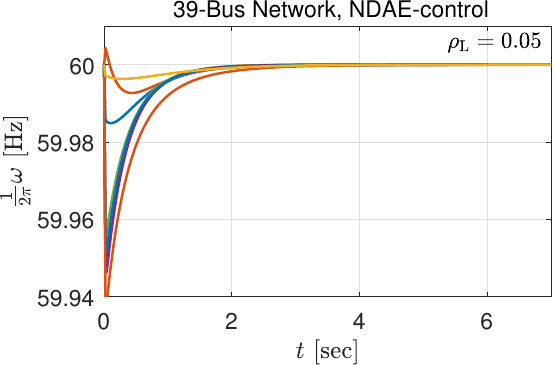}}{}\hspace{-0.08cm}\vspace{-0.25cm}
	\subfloat[]{\includegraphics[keepaspectratio=true,scale=0.627]{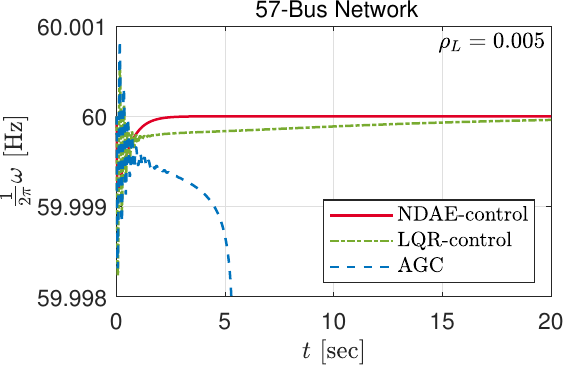}}{}\hspace{-0.08cm}
	\subfloat[]{\includegraphics[keepaspectratio=true,scale=0.627]{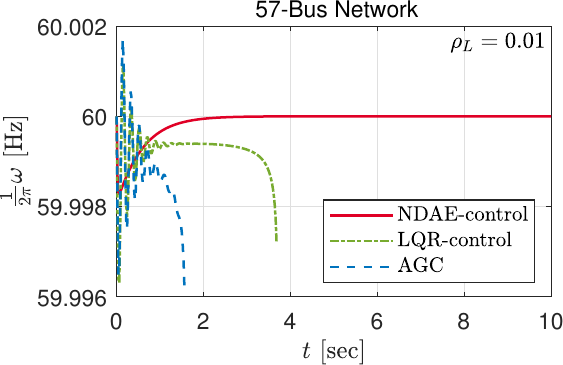}}{}\hspace{-0.08cm}	\subfloat[]{\includegraphics[keepaspectratio=true,scale=0.627]{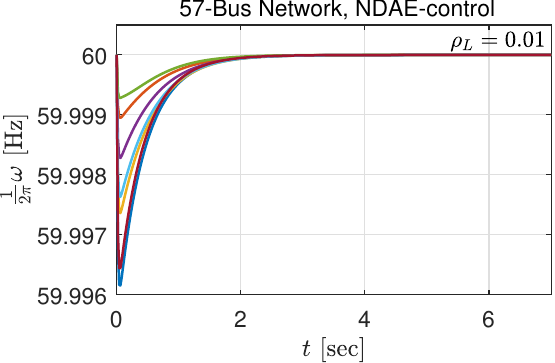}}{}\hspace{-0.08cm}
	\vspace{-0.1cm}
	\caption{Numerical simulation results: Figs. (a), (b), (c) illustrate the frequency of Generator 1 for the $9$-bus network; Figs. (d), (e), (f) illustrate the frequency of Generator 1 for the $14$-bus network; Figs. (g) and (h) illustrate the frequency of Generator 1 while Fig. (i) shows all generators' frequency for the $39$-bus network with the NDAE-control; Figs. (j) and (k) illustrate the frequency of Generator 1 while Fig. (l) shows all generators' frequency for the $57$-bus network with the NDAE-control. Although the trajectories of the rotor frequency for the $39$-bus network given in Figs. (g) and (h) seem to converge, the AGC actually fails to stabilize the system for $\rho_{\mr{L}} = 0.01$ and $\rho_{\mr{L}} = 0.05$ while the LQR-control cannot stabilize the system when $\rho_{\mr{L}} = 0.05$.} %
	\label{fig:scenario1}\vspace{-0.1cm}
	\vspace{-0.3cm}
\end{figure*}

\begin{figure}
	\vspace{-0.1cm}
	\centering 
	\subfloat[\label{fig:scenario1_case14_low_Pg}]{\includegraphics[keepaspectratio=true,scale=0.627]{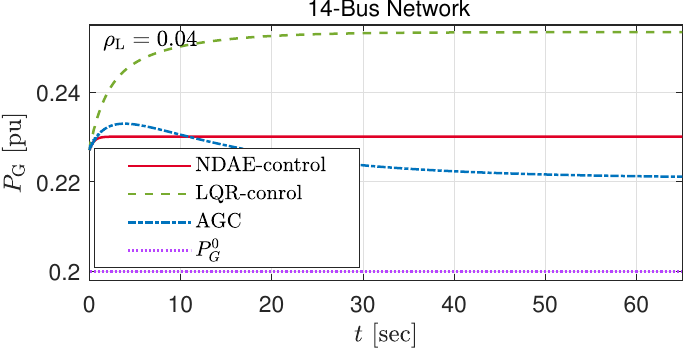}}{}{}\vspace{-0.25cm}
	\subfloat[\label{fig:scenario1_case14_low_v}]{\includegraphics[keepaspectratio=true,scale=0.627]{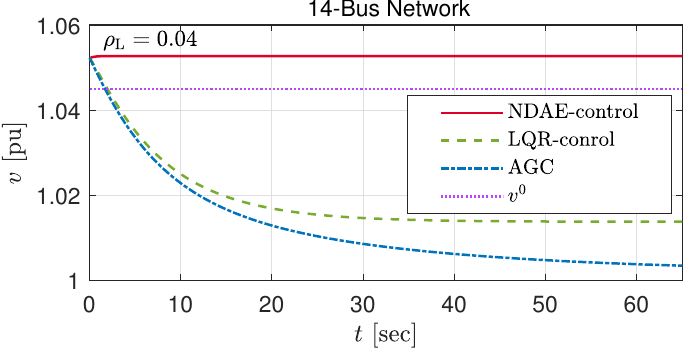}}{}{}\hspace{-0.1cm}
	\vspace{-0.2cm}
	\caption{The trajectories of active power produced by Generator 5 (a) and the modulus of the voltage at Bus 2 (b) for the $14$-bus power network with $\rho_{\mr{L}} = 0.04$. The notations $P_{G}^0$ and $v^0$ represent the corresponding initial steady-state values before disturbance is applied to the network.}
	\label{fig:scenario1_case14_low}
	\vspace{-0.2cm}
\end{figure}

\subsection{Parameters and Setup for Numerical Simulations}\label{ssec:setup}

This section presents numerical simulations for investigating the performance of the proposed approach in stabilizing several IEEE test networks with respect to load and renewable disturbances. Every numerical simulation is performed
using MATLAB R2020b running on a 64-bit Windows 10 with
a 3.0GHz AMD Ryzen\textsuperscript{TM} 9 4900HS processor and 16 GB of RAM, whereas all convex optimization problems are solved through YALMIP \cite{Lofberg2004} optimization interface along with MOSEK \cite{Andersen2000} solver. All dynamical simulations for NDAEs are performed using MATLAB's index-one DAEs solver \texttt{ode15i}. Four power networks are considered in this study: 
\begin{itemize}[leftmargin=*]
	\item \textsl{$9$-bus network}: The Western System Coordinating Council (WSCC) $9$-bus system with $3$ synchronous generators. 
	\item \textsl{$14$-bus network}:  Consisting of $14$-bus system with $5$ synchronous generators, representing a portion of the American Electric Power System (AEPS) in the Midwestern US.
	\item \textsl{$39$-bus network}: Represents the New England $10$-machine, $39$-bus system.
	\item \textsl{$57$-bus network}: Consisting of $57$ buses with $7$ synchronous generators, which again represents a part of the AEPS.
\end{itemize}
In this study, the loads are presumed to be of constant power type while renewable power plants---such as wind farms and solar PVs---are modeled as loads with \textit{negative} power, thereby injecting active power into the network. For the $9$-bus and $14$-bus networks, every load bus is connected to one renewable power plant. For the $39$-bus and $57$-bus networks, one renewable power plant is attached to a load bus when the consumed power is equal to or exceeds $3\,\mathrm{pu}$ and $0.1\,\mathrm{pu}$, respectively. The initial conditions as well as steady-state values of the power network before disturbance is applied are computed from the solutions of power flow, which is obtained from MATPOWER \cite{Zimmerman2011MATPOWER} function \texttt{runpf}. The power base 
is chosen to be $100$ MVA.
The generator parameters
are obtained from Power System Toolbox (PST) \cite{sauer2017power}, where the regulation and chest time constants are set to $R_{\mr{D}i} = 0.02\;\mathrm{Hz/pu}$ and $T_{\mr{CH}i} = 0.2\;\mathrm{sec}$ for all $i\in\mathcal{G}$. 

\setlength{\textfloatsep}{10pt}
\begin{table}[t]
	\scriptsize
	\vspace{-0.1cm}
	\centering 
	\caption{The comparison of the total rotor speed deviations with respect to different levels of disturbance taken at $t = 15\;\mathrm{sec}$, except for the $9$-bus network where $t = 10\;\mathrm{sec}$.
		The dash symbol "$-$" indicates that the rotor speed does not converge and bold numbers indicate the minimum values.}
	\label{tab:comp_speed_deviation}
	\vspace{-0.15cm}
	\renewcommand{\arraystretch}{1.65}
	\begin{threeparttable}
		\begin{tabular}{|c|c|c|c|c|}
			\hline
			\multirow{2}{*}{Network} & \multirow{2}{*}{$\rho_{\mr{L}} = -\rho_{\mr{R}}$} & \multicolumn{3}{c|}{$\left(\norm{\omega_{0}\times \m 1-\m \omega(\tilde{t}_k)}_2\times 10^{3}\right)$}                                                 \\ \cline{3-5} 
			&    &   NDAE-control            &   LQR-control    &  AGC  \\ \hline \hline
			\multirow{3}{*}{$9$-bus} &      $0.04$           &  $\mathbf{0.177}$     &  $1.656$     & $1.575$      \\ \cline{2-5} 
			&      $0.08$  &  $\mathbf{0.354}$     &  $3.538$     &   $-$      \\ \cline{2-5} 
			&      $0.12$               &   $\mathbf{0.543}$  &    $-$    &   $-$     \\ \hline
			\multirow{3}{*}{$14$-bus} &      $0.04$          &  $\mathbf{6.437\times 10^{-7}}$     &   $3.046$    &  $3.406$     \\ \cline{2-5} 
			&       $0.08$          &  $\mathbf{9.409\times 10^{-6}}$     &    $6.152$   &   $-$   
			\\ \cline{2-5} 
			&     $0.12$          &  $\mathbf{4.433\times 10^{-6}}$ &   $-$    &  $-$       \\ \hline
			\multirow{2}{*}{$39$-bus} &          $0.01$      &  $\mathbf{2.027\times 10^{-5}}$ & $10.820$ & $-$  \\ \cline{2-5} 
			&   $0.05$    & $\mathbf{5.130\times 10^{-6}}$  & $-$  &  $-$  \\ \hline
			\multirow{2}{*}{$57$-bus} &    $0.005$     &  $\mathbf{8.092\times 10^{-6}}$     &   $1.534$  &    $-$   \\ \cline{2-5} 
			&    $0.01$    & $\mathbf{1.592\times 10^{-5}}$  &   $-$  &   $-$  \\ \hline
		\end{tabular}
	\end{threeparttable}
	\vspace{-0.15cm}
\end{table}
\setlength{\floatsep}{10pt}

\begin{figure*}
	\vspace{-0.4cm}
	\centering 
	\subfloat[]{\includegraphics[keepaspectratio=true,scale=0.627]{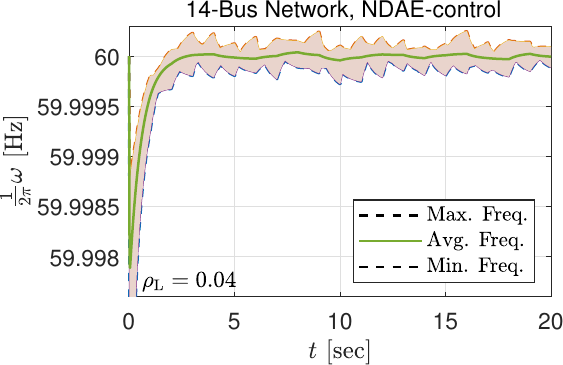}}{}\hspace{-0.08cm}
	\subfloat[]{\includegraphics[keepaspectratio=true,scale=0.627]{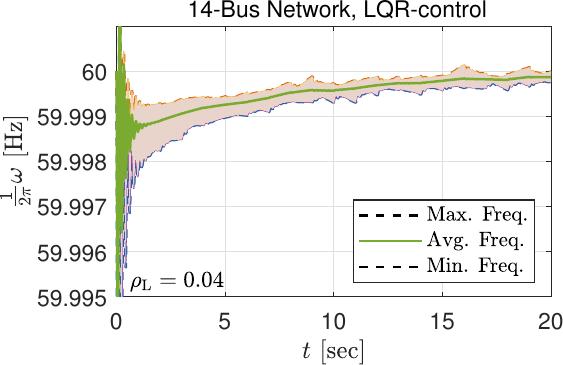}}{}\hspace{-0.08cm}
	\subfloat[]{\includegraphics[keepaspectratio=true,scale=0.627]{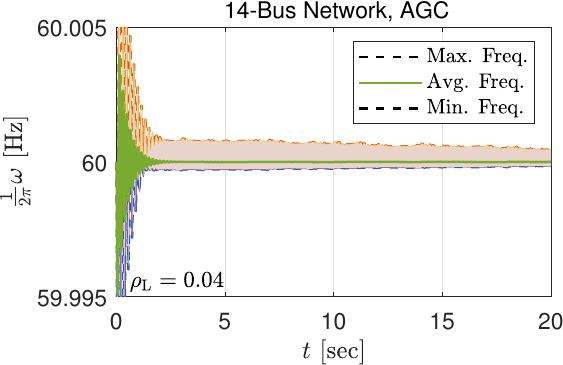}}{}\hspace{-0.08cm}\vspace{-0.2cm}
	\subfloat[]{\includegraphics[keepaspectratio=true,scale=0.627]{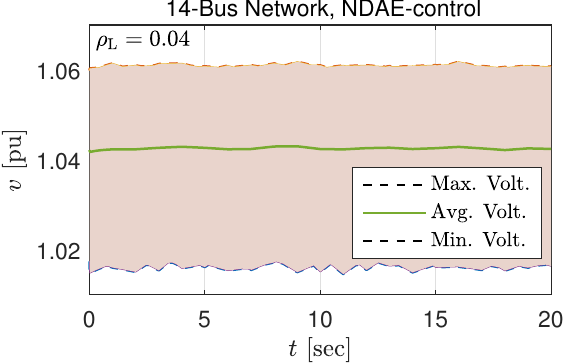}}{}\hspace{-0.08cm}
	\subfloat[]{\includegraphics[keepaspectratio=true,scale=0.627]{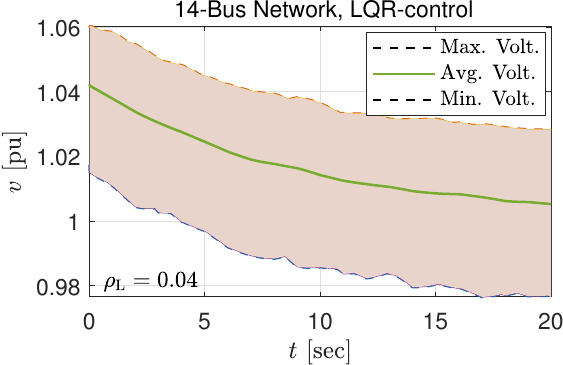}}{}\hspace{-0.08cm}
	\subfloat[]{\includegraphics[keepaspectratio=true,scale=0.627]{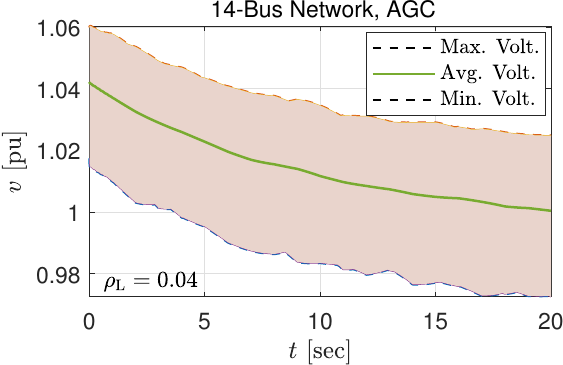}}{}\hspace{-0.08cm}\vspace{-0.15cm}
	\vspace{-0.0cm}
	\caption{Numerical simulation results for the $14$-bus network with renewables uncertainty: Figs. (a), (b), (c) illustrate the overall frequency figure of all generators while Figs. (d), (e), (f) illustrate the overall modulus of bus voltage for all buses using the NDAE-control, LQR-control, and AGC, respectively.} %
	\label{fig:scenario2}\vspace{-0.1cm}
\end{figure*}

\vspace{-0.0cm} 
\subsection{LRFC Under Different Levels of Step Disturbances}\label{ssec:LRFC_numerical_test}
Herein, we analyze the performance of the proposed control strategy---which is referred to as \textit{NDAE-control}---in performing LRFC for the aforementioned power network test cases against two control strategies prominent in power systems literature, namely the Automatic Generation Control (AGC) and Linear Quadratic Regulator (LQR) control (referred to as \textit{LQR-control}). We do not compare our method with the ones proposed in \cite{siljak2002robust,Elloumi2002,Marinovici2013} since these methods are designed for the simplified nonlinear ODE model of power networks, and thus are not applicable for performing LRFC using the model given in \eqref{eq:nonlinearDAEexplicit}. 
The controller gain for the NDAE-control is obtained from solving problem $\mathbf{P}$.
Since the form of nonlinearities in $\m f_d(\cdot)$ and $\m f_a(\cdot)$ are much more complex than the ones in \cite{siljak2002robust,Marinovici2013}, the associated bounding matrices are instead chosen to be $$\left(\m H_d^d\right)^2 = \m I,\;\left(\m H_a^d\right)^2 = \m I,\;\left(\m H_d^a\right)^2 = \m I,\;\left(\m H_a^a\right)^2 = \m I,$$ for the $9$-bus and $14$-bus networks while the following values
$$\left(\m H_d^d\right)^2 = 10\m I,\;\left(\m H_a^d\right)^2 = 10\m I,\;\left(\m H_d^a\right)^2 = 10\m I,\;\left(\m H_a^a\right)^2 = 10\m I,$$ are selected for the $39$-bus and $57$-bus networks. The bounding matrices for the $39$-bus and $57$-bus networks are set to be larger than those for the $9$-bus and $14$-bus networks since the $39$-bus and $57$-bus networks are comprised of significantly larger nodes and interconnections.
For the AGC, it is implemented based on the method described in \cite{Taha2019TCNS,wollenbergbook2012}, where it provides a set of control inputs for the governor reference signals only. 
The AGC calculates such input signals by adding an extra dynamic state $\chi$ to the power network model \eqref{eq:nonlinearDAEexplicit}, specified as
\begin{align}
\dot{\chi} &= K_{\mathrm{G}}\left(  -\chi - \mathrm{ACE}+ \sum\limits_{i = 1}^G (P_{\mr{G}i}-P_{\mr{G}i}^0)\right), \label{eq:AGC_dynamics}
\end{align}
where $K_{\mr{G}}$ is an integrator gain for the AGC dynamics, which value is set to be $1000$, and $P_{\mr{G}i}^0$ is  the $i$-th steady-state generator active power before disturbance. The term $\mathrm{ACE}$ in \eqref{eq:AGC_dynamics} stands for \textit{area control error} and defined as \cite{Taha2019TCNS}
\begin{align*}
	\mr{ACE}&:= \frac{1}{G}\sum\limits_{i =1}^G(\frac{1}{R_{\mr{D}i}}+D_i) (\omega_i - \omega_0). 
\end{align*}
Following \cite{Wang2019TPWRS}, each power network is treated as a single area. The governor reference signal for each generator $i\in\mc{G}$ is given as $T_{\mr{CH}i} =  T_{\mr{CH}i}^0+K_{i} \chi$, where $K_i:=P_{\mr{G}i}/\sum_{i =1}^G P_{\mr{G}i}$, for every $i\in\mc{G}$, indicates the participation factor of each generator such that $\sum_{i =1}^G K_{i} = 1$, and $T_{\mr{CH}i}^0$ is the corresponding steady-state governor reference signal before disturbance. However, since AGC only provides value for $T_{\mr{CH}i}$, the control inputs for the internal field voltage are calculated with the aid of LQR-control. It is important to mention that the controller gain for LQR-control is retrieved from solving the corresponding LMI 
specified in Theorem 1 of \cite{Khlebnikov2015}, which is reliant on the \textit{linearized} dynamics corresponding to the initial operating point.

The numerical simulation is performed as follows. Initially, the system operates with total load of $\left(P_{\mr{L}}^0,Q_{\mr{L}}^0\right)$ and total generated power from renewables of $\left(P_{\mr{R}}^0,Q_{\mr{R}}^0\right)$. For each of the power network test cases, the following values are chosen: $P_{\mr{L}}^0+jQ_{\mr{L}}^0 = 3.15 + j1.15\;\mathrm{pu}$ and $P_{\mr{R}}^0+jQ_{\mr{R}}^0 =  0.63\;\mathrm{pu}$ for the {$9$-bus network}, $P_{\mr{L}}^0+jQ_{\mr{L}}^0 = 3.15 + j1.15\;\mathrm{pu}$ and $P_{\mr{R}}^0+jQ_{\mr{R}}^0 =  0.63\;\mathrm{pu}$ for the {$14$-bus network}, $P_{\mr{L}}^0+jQ_{\mr{L}}^0 = 62.5423 + j13.871\;\mathrm{pu}$ and $P_{\mr{R}}^0+jQ_{\mr{R}}^0 =  8.1712\;\mathrm{pu}$ for the {$39$-bus network}, while $P_{\mr{L}}^0+jQ_{\mr{L}}^0 = 12.508
 + j3.364\;\mathrm{pu}$ and $P_{\mr{R}}^0+jQ_{\mr{R}}^0 =  2.2888\;\mathrm{pu}$ for the {$57$-bus network}. 
 Immediately after $t > 0$, the loads and renewables are experiencing an abrupt step change in the amount of consumed and produced power, which triggers the system to depart from its initial equilibrium point. The new value of complex power for loads and renewables are specified as $P_{\mr{L}}^e+jQ_{\mr{L}}^e := (1 + \rho_{\mr{L}})(P_{\mr{L}}^0+jQ_{\mr{L}}^0)$ and $P_{\mr{R}}^e+jQ_{\mr{R}}^e := (1 + \rho_{\mr{R}})(P_{\mr{R}}^0+jQ_{\mr{R}}^0)$ where $\rho\in\mbb{R}$ determines the quantity of the disturbance. In this numerical simulation, we consider different levels of disturbance: $\rho_{\mr{L}} = 0.04$, $\rho_{\mr{L}} = 0.08$, and $\rho_{\mr{L}} = 0.12$ for the {$9$-bus network} and {$14$-bus network}, $\rho_{\mr{L}} = 0.01$, and $\rho_{\mr{L}} = 0.05$ for the {$39$-bus network}, and $\rho_{\mr{L}} = 0.005$ and $\rho_{\mr{L}} = 0.01$ for the {$57$-bus network}. For the disturbance coming from renewables, we select $\rho_{\mr{R}} = -\rho_{\mr{L}}$. 
 
The results of the numerical simulation are illustrated in Fig. \ref{fig:scenario1}. For the $9$-bus network, the proposed NDAE-control is able to stabilize the system even when the disturbance is considerably high ($12\%$ for this network). This is in contrast to the AGC and LQR-control, as they are only able to maintain stability with relatively low ($4\%$) and moderate ($8\%$) disturbances. Similar behavior is also observed from the simulation results for the $14$-bus, $39$-bus, and $57$-bus networks: the LQR-control is not able to maintain frequency stability when the disturbance achieves $12\%$, $5\%$, and $1\%$ while the AGC fails even with $8\%$, $1\%$, and $0.5\%$ disturbance, respectively, for the $14$-bus, $39$-bus, and $57$-bus networks. It can be seen from Fig. \ref{fig:scenario1} that the frequency trajectories due to the NDAE-control converge rapidly to the synchronous frequency $\omega_0$, unlike the other controllers.
Table \ref{tab:comp_speed_deviation} presents the norm of rotor speed deviations for all generators with respect to various levels of disturbance. 
It is evident that the NDAE-control can provide stabilization for the power networks with a decent convergence rate. 
It is also observed that each controller brings the system's operating point to a new equilibrium---this can be seen from the trajectories of active power and bus voltage for the $14$-bus network with low disturbance as shown in Fig. \ref{fig:scenario1_case14_low}. 

 \begin{figure}
		\vspace{-0.1cm}
	\centering 
	\subfloat[\label{fig:sparsity_case9}]{\includegraphics[keepaspectratio=true,scale=0.24]{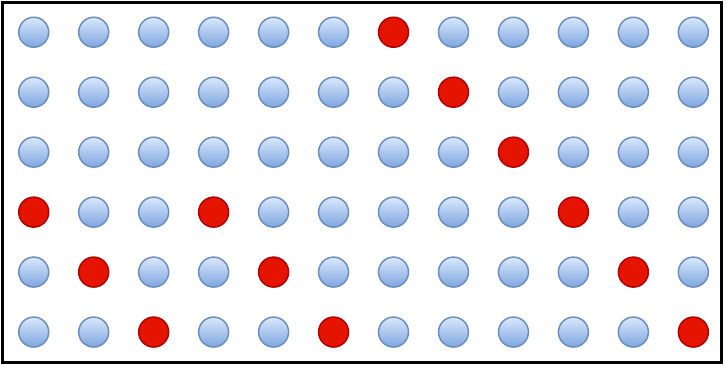}}{}{}\vspace{-0.2cm}
	\subfloat[\label{fig:sparsity_case14}]{\includegraphics[keepaspectratio=true,scale=0.24]{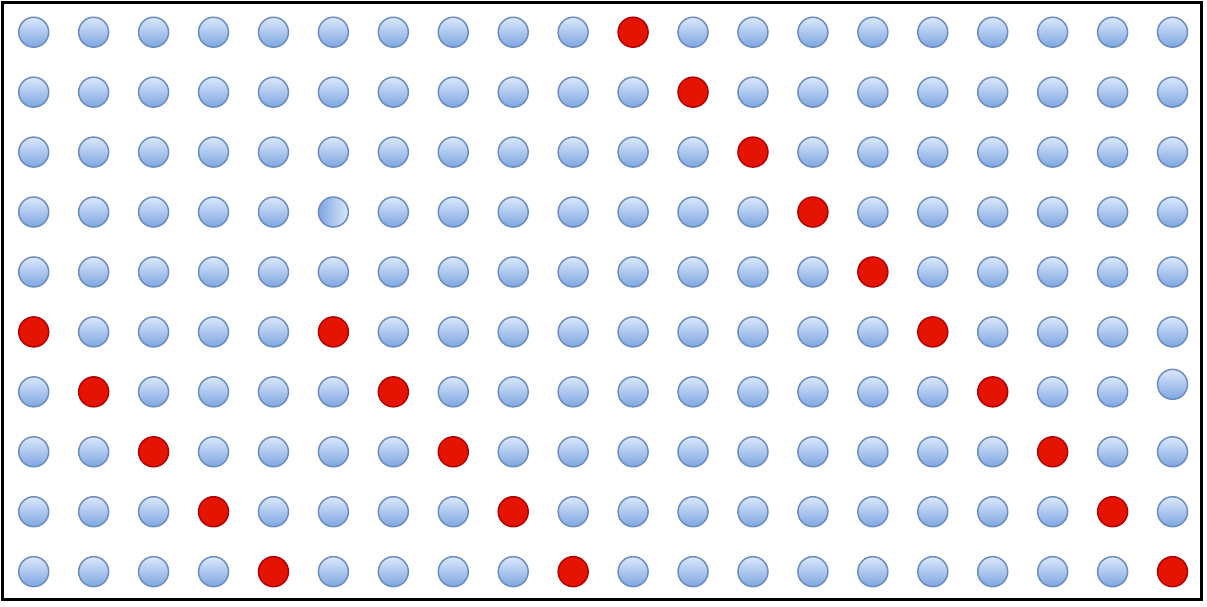}}{}{}\hspace{-0.1cm}
	\caption{Sparsity patterns of the controller gain matrix $\m K_d$ for the $9$-bus (a) and $14$-bus (b) networks. The red circles represent entries with significant magnitudes. Similar patterns are also found on the remaining larger networks.}
	\label{fig:sparsity_pattern}
	\vspace{-0.2cm}
\end{figure}

 \vspace{-0.2cm}
\subsection{Assessment Against Renewable Generation Uncertainties}\label{ssec:renewables_disturbance} 
In this section, we study the $14$-bus network while injecting the generated power from renewables with random Gaussian noise $z_i(t)$ with zero mean and variance of $0.01(P_{\mr{Ri}}^0+jQ_{\mr{Ri}}^0)$ for each $i\in \mc{R}$ such that $$P_{\mr{Ri}}^e+jQ_{\mr{Ri}}^e := (1 + \rho_{\mr{R}})(P_{\mr{Ri}}^0+jQ_{\mr{Ri}}^0) + (1+j) z_i(t),\;\;\forall i\in \mc{R}.$$
To compensate for the random noise, the simulation is performed $10$ times and the resulting outcomes are averaged.
The results of this numerical simulation with low step disturbance $\rho_{\mr{L}} = 0.04$ are illustrated in Fig. \ref{fig:scenario2}, from which it can be seen that the maximum and minimum frequency deviations for the NDAE-control are experiencing much mode fluctuations compared to those from the LQR-control and AGC. The NDAE-control is able to maintain generators' frequency close to $60\;\mr{Hz}$ without exhibiting significant oscillations.
It is also indicated from this figure that, for the NDAE-control, the average bus voltage across the network has a roughly flat profile. This result can be attributed to the centralized control structure in the LQR-control and AGC, while
the proposed DAE-control implements a decentralized control framework---discussed in Section \ref{ssec:comp_time_sparsity}.

 \vspace{-0.20cm}
\subsection{On The Controller Gain's Sparsity Structure}\label{ssec:comp_time_sparsity} 
A decentralized control is much preferable to a centralized control since in the former type of control, stabilization can be maintained using local measurements only. As such, our NDAE-control is more practical than AGC and LQR since the NDAE-control implements a decentralized control structure---this is indicated by the certain sparsity pattern on the feedback gain matrix $\m K_d$. 
The patterns for the $9$-bus and $14$-bus networks are described in Fig. \ref{fig:sparsity_pattern}. The small red circles denote entries with significant magnitudes, i.e., entries whose magnitudes are greater or equal to $10^{-6}$. Notice that the dynamic states are ordered as ${\m x}_d := \bmat{\m \delta^\top\;\;\m \omega^\top\;\;\m E'^\top\;\;\m T_{\mr{M}}^\top}^\top$ according to Section \ref{sec:modeling}. Based on this ordering, the patterns depicted in Fig. \ref{fig:sparsity_pattern} suggest that 
the inputs for each generator can be constructed from local measurements (or estimation) of its internal states. The decentralized control structure allows the internal field voltage to be constructed by $E_{\mr{fd}i} = K_{D(i,2G+i)} E'_i$ while the governor reference signal to be given by 
\begin{align*}
	T_{\mr{r}i} = K_{D(G+i,i)}\delta_i + K_{D(G+i,G+i)}\omega_i+K_{D(G+i,3G+i)}T_{\mr{M}i},
\end{align*} 
for all $i\in \mc{G}$ where $K_{D(i,j)}$ is the $(i,j)-$th element of $\m K_d$. 
The sparsity structure of  $\m K_d$ is suspected to be caused by the use of \eqref{eq:nonlinearDAEpertGen} when the matrix $\m K_d$ is synthesized for the NDAE-control since the NDAE model in \eqref{eq:nonlinearDAEpertGen} retains the structure of the power network while, in contrast, this structure is lost in the linearized power network's model used in AGC and LQR.   

 \vspace{-0.0cm}
\section{Summary and Future Directions}\label{sec:conclusion}
A novel approach for LRFC in multi-machine power networks is proposed. In contrast to other methods from the literature, our approach is based on the NDAE representation of power networks and accordingly, we develop a computational approach based on LMI to construct the stabilizing controller gain matrix. The proposed approach stands out in the following manner: \textit{(a)} its independence from any linearization around any operating points, \textit{(b)} the resulting controller gain matrix can sufficiently maintain the system's frequency around the desired equilibrium against significant disturbances originating from the loads and renewables, and  \textit{(c)} although our approach relies on advanced DAE systems theory, the proposed LRFC strategy is as simple as proportional decentralized control framework and therefore, can be implemented to large-scale power systems without the need for any special tools.

In our future work, we are planning to \textit{(i)} extend the proposed NDAE-control and develop a robust control method to handle adverse impacts caused by parametric uncertainties, \textit{(ii)} investigate the cause of decentralized sparsity patterns in the controller
gain resulting from the NDAE-control, and \textit{(iii)} study the controller's applicability to perform wide-area damping control in inverter-based, renewables-heavy power networks.

\bibliographystyle{IEEEtran}
\bibliography{bib_file}

\vspace{-0.4cm}
\appendices

\section{Description of Matrices in NDAEs \eqref{eq:nonlinearDAEexplicit}}\label{appdx:A}
\vspace{-0.1cm}
The matrix ${\m A}_d$ is constructed as
\begin{align*}
	\bmat{\m O &\m I & \m O & \m O \\ \m O & -\Diag\left(\m D \oslash \m M\right)& \m O & \Diag\left(\m 1\oslash \m M\right) \\\m O &\m O & -{\m A}_{d(3,3)} & \m O \\ \m O & {\m A}_{d(4,2)} &\m O & -\Diag\left(\m 1\oslash \m T_{\mr{CH}}\right)},
\end{align*}
in which the two submatrices in ${\m A}_d$ are given as
\begin{align*}
	{\m A}_{d(3,3)} &= \Diag\left(\m x_{\mr{d}}\oslash \left(\m x_{\mr{d}}'\odot \m T'_{\mr{d0}}\right) \right) \\  {\m A}_{d(4,2)} &= \Diag\left(\m 1\oslash \left(\m R_{\mr{d}}\odot \m T_{\mr{CH}}\right) \right),
\end{align*}
and the matrices ${\m G}_d$, ${\m B}_d$, ${\m F}$ are specified as
\begin{align*}
	{\m G}_d &:= \bmat{\m O&\m O\\\Diag\left(\m 1\oslash \m M\right)&\m O\\\m O&\Diag\left((\m x_{\mr{d}}-\m x_{\mr{d}}')\oslash \left(\m x_{\mr{d}}'\odot \m T'_{\mr{d0}}\right) \right)\\\m O&\m O} \\
	{\m B}_d &:= \bmat{\m O&\m O\\\m O&\m O\\\Diag\left(\m 1\oslash \m T'_{{\mr{d0}}}\right)&\m O\\\m O&\Diag\left(\m 1\oslash \m T_{{\mr{CH}}}\right)}\\ {\m h} &:=\bmat{\m 1 \\ \m D \oslash \m M \\ \m O \\ \m 1\oslash \left(\m R_{\mr{d}}\odot \m T_{\mr{CH}}\right)}.
\end{align*}
The function ${\m f}_d(\cdot)$ in \eqref{eq:nonlinearDAEexplicit} is given as
\begin{align*}
	{\m f}_d\left({\m x}_d,{\m x}_a\right) := \bmat{\m P_{\mr{G}}\\ \{v_i\cos(\delta_{i}-\theta_i)\}_{i\in \mc{G}} },
\end{align*}
Next, the matrices $\m A_a$, $\m G_a$, and $\m B_a$ are detailed as
\begin{align*}
	\m A_a = \bmat{-\m I&\m O\\ {\m A}_p&\m O},\;\m G_a = \Blkdiag\left(\tilde{\m G}_{a},\m I\right),\;\m B_a = \bmat{\m O\\ {\m B}_p},
\end{align*} 
where ${\m A}_p := \bmat{-\m I \;\;\;\m O}^\top$ and $\tilde{\m G}_a$ is detailed as
\begin{align*}
	\tilde{\m G}_a &:= \bmat{\Diag\left(\m 1\oslash \m x_{\mr{d}}'\right)&\m O\\\tilde{\m G}_{a(1,2)}&\m O\\\m O&\Diag\left(\m 1\oslash \m x_{\mr{d}}'\right)\\\m O&-\tilde{\m G}_{a(2,4)}\\ \m O & \tilde{\m G}_{a(2,5)}}^\top,
\end{align*}
with $\tilde{\m G}_{a(1,2)} := \Diag\left((\m x_{\mr{d}}'-\m x_{\mr{q}})\oslash \left(2\m x_{\mr{d}}'\odot \m x_{\mr{q}}\right) \right)$,  $ \tilde{\m G}_{a(2,4)} := \Diag\left((\m x_{\mr{d}}'+\m x_{\mr{q}})\oslash \left(2\m x_{\mr{d}}'\odot \m x_{\mr{q}}\right) \right)$, and $ \tilde{\m G}_{a(2,5)} := \tilde{\m G}_{a(1,2)}$. The matrix ${\m B}_p$ is a binary matrix having $1$ in each of its elements corresponding to buses that are connected with renewables and/or load---the entries of ${\m B}_p$ are set to be zero otherwise. The function ${\m f}_a\left(\cdot\right)$ in \eqref{eq:nonlinearDAEexplicit} is constructed as
\begin{align*}
{\m f}_a\hspace{-0.05cm}\left({\m x}_d,{\m x}_a\right) \hspace{-0.075cm}:=\hspace{-0.075cm} \bmat{\{E'_{i}v_i\sin(\delta_i-\theta_i)\}_{i\in \mc{G}} \\\{v_i^2\sin(2(\delta_i-\theta_i))\}_{i\in \mc{G}} \\\{E'_{i}v_i\cos(\delta_i-\theta_i)\}_{i\in \mc{G}} \\ \{v_i^2\}_{i\in \mc{G}} \\ \{v_i^2\cos(2(\delta_i-\theta_i))\}_{i\in \mc{G}} \\ \left\{ \hspace{-0.05cm}\sum_{j=1}^{N}\hspace{-0.05cm} v_iv_j\hspace{-0.05cm}\left(G_{ij}\cos \theta_{ij} \hspace{-0.05cm}+ \hspace{-0.05cm}B_{ij}\sin \theta_{ij}\right)\right\}_{i\in \mc{G}\hspace{-0.05cm}} \\ \left\{\hspace{-0.05cm}\sum_{j=1}^{N}\hspace{-0.05cm} v_iv_j\hspace{-0.05cm}\left(G_{ij}\sin \theta_{ij} \hspace{-0.05cm}- \hspace{-0.05cm}B_{ij}\cos \theta_{ij}\right)\right\}_{i\in \mc{G}\hspace{-0.05cm}} \\
	\left\{ \hspace{-0.05cm}\sum_{j=1}^{N}\hspace{-0.05cm} v_iv_j\hspace{-0.05cm}\left(G_{ij}\cos \theta_{ij} \hspace{-0.05cm}+ \hspace{-0.05cm}B_{ij}\sin \theta_{ij}\right)\right\}_{i\in \mc{N}\setminus\mc{G}\hspace{-0.05cm}} \\ \left\{\hspace{-0.05cm}\sum_{j=1}^{N}\hspace{-0.05cm} v_iv_j\hspace{-0.05cm}\left(G_{ij}\sin \theta_{ij} \hspace{-0.05cm}- \hspace{-0.05cm}B_{ij}\cos \theta_{ij}\right)\right\}_{i\in \mc{N}\setminus\mc{G}\hspace{-0.05cm}}}\hspace{-0.1cm}.
\end{align*} 

\section{Proof of Theorem \ref{thm:exp_stabilization}}\label{appdx:B}
\vspace{-0.1cm}
The following lemma is presented first due to its importance in the proof of Theorem \ref{thm:exp_stabilization}.
\vspace{-0.1cm}
\begin{mylem}\label{lem:lemma1}
	For any matrix $\m M\in \mbb{R}^{r\times s}$ with $r < s$ and scalars $a,b\in \mbb{R}_{++}$, the following holds
	\begin{align}
		\m M^\top (a \m M \m M^\top + b \m I)^{-1} \m M -a \m I\preceq 0. \label{eq:lemma1}
	\end{align}
\end{mylem}
\begin{proof}
Consider the singular value decomposition of $\m M$ written as $\m M = \m U \bmat{\m \Lambda\;\; \m O} \m V^\top$ where $\m \Lambda \in \mbb{R}^{r\times r}$ is a diagonal matrix populating all singular values of $\m M$ while $\m U\in \mbb{R}^{r\times r}$ and $\m V\in \mbb{R}^{s\times s}$ are two orthogonal matrices. As the term $a \m M \m M^\top + b \m I$ for positive scalars $a$ and $b$ can be written as 
\begin{align*}
	a \m M \m M^\top + b \m I = \m U \left(a\m \Lambda^2+b\m I\right)\m U^\top,
\end{align*}
then it can be shown that the term $\m M^\top (a \m M \m M^\top + b \m I)^{-1} \m M$ is equal to $$\m V\left(\mathrm{Blkdiag}\left(a\m \Lambda^2\left(\m \Lambda^2+\tfrac{b}{a}\m I\right)^{-1},\m O\right)\right)\m V^\top.$$
Nevertheless, since the inequality $\left(\m \Lambda^2+\tfrac{b}{a}\m I\right)^{-1} \preceq \m \Lambda^{-2}$ implies $a\m \Lambda^2\left(\m \Lambda^2+\tfrac{b}{a}\m I\right)^{-1} \preceq a \m I$, \eqref{eq:lemma1} is inferred.
\end{proof}
\vspace{-0.1cm}
Now we are ready to prove Theorem \ref{thm:exp_stabilization}, which is decomposed into four parts: 
\begin{enumerate}[label=(\alph*)]
	\item Showing that the dynamic state is asymptotically stable.
	\item Demonstrating that the matrices associated with the Lyapunov function are nonsingular.
	\item Showing that the algebraic state is asymptotically stable.
	\item Establishing the matrix inequalities in \eqref{eq:LMI_stabilization_all}.
\end{enumerate}
\noindent	 	\textit{\textbf{(a)}:} Let $V: \mathbb{R}^{n_d}\rightarrow  \mathbb{R}_+$ be a Lyapunov function candidate such that $V(t) = \check{\m x}_d^\top \m E_d^\top \m P_1 \check{\m x}_d$ where $\m P_1\in \mbb{R}^{n_d\times n_d}$ is assumed (for now) to be nonsingular and $\m E_d^\top \m P_1 =\m P_1^\top\m E_d \succ 0$. The time derivative of $V(\cdot)$ is equivalent to
\begin{align}
	\begin{split}
		\dot{V}(t) 	&= \left(\bar{\m A}_d \check{\m x}_d + \m G_d \check{\m f}_d({\m x},{\m x}^e)\right)^\top \hspace{-0.1cm}\m P_1 \check{\m x}_d \\
		&\quad + \check{\m x}_d^\top \m P_1^\top \left(\bar{\m A}_d \check{\m x}_d + \m G_d \check{\m f}_d({\m x},{\m x}^e)\right),
	\end{split} \label{eq:thm1_proof_1}
\end{align} 
where $\bar{\m A}^d_d := {\m A}_d +  {\m B}_d \m K_d.$
For any DAE of index $H$, then for any function $\m \Gamma_i(\cdot)$, $i \in \{0,1,\hdots,H-1\}$,  we have \cite{Franco2020}
\begin{align}
	\hspace{-0.2cm}	\sum^{H-1}_{i=0}\m\Gamma_i(\check{\m x}_d,\check{\m x}_a)\dfrac{d^i \m h(\check{\m x}_d,\check{\m x}_a)}{d t^i} = 0, \;\forall \check{\m x}_d \in \mathbfcal{X}_d,\,\check{\m x}_a \in \mathbfcal{X}_a,\label{eq:thm1_proof_2}
\end{align}
where the function $\m h(\cdot)$ represents all terms in right-hand side of \eqref{eq:nonlinearDAEpertGen-2}. Since the DAE is of index-one, thanks to Assumption \ref{asmp:index_one}, then the following choice of $\m \Gamma_0(\check{\m x}_d,\check{\m x}_a)$ such that
\begin{align} \m \Gamma_0(\check{\m x}_d,\check{\m x}_a) := \check{\m x}_d^\top \m P_2^\top + \check{\m x}_a^\top \m P_3^\top,\label{eq:thm1_proof_2-b}
\end{align}
for some $\m P_2\in\mbb{R}^{n_a\times n_d}$ and $\m P_3\in\mbb{R}^{n_a\times n_a}$ is sufficient. 
Adding \eqref{eq:thm1_proof_2} to \eqref{eq:thm1_proof_1}, using \eqref{eq:thm1_proof_2-b}, allows \eqref{eq:thm1_proof_1} to be expressed into
\begin{align}
	\begin{split}
		\hspace{-0.00cm}\dot{V}&(t)\hspace{-0.05cm}=\hspace{-0.05cm}\left(\bar{\m A}_d \check{\m x}_d + \m G_d \check{\m f}_d({\m x},{\m x}^e)\right)^\top \hspace{-0.1cm}\m P_1 \check{\m x}_d \\
		\hspace{-0.2cm}&\quad + \check{\m x}_d^\top \m P_1^\top \left(\bar{\m A}_d \check{\m x}_d + \m G_d \check{\m f}_d({\m x},{\m x}^e)\right)\\
		\hspace{-0.2cm}&\quad + \left({\m A}_a {\m x}_a + {\m G}_a \check{\m f}_a\left({\m x},{\m x}^e\right)\right)\hspace{-0.1cm}^\top \hspace{-0.00cm}(\m P_2 \check{\m x}_d + \m P_3 \check{\m x}_a) \\
		\hspace{-0.2cm}&\quad +\left (\check{\m x}_d^\top \m P_2^\top + \check{\m x}_a^\top \m P_3^\top\right)\left({\m A}_a {\m x}_a + {\m G}_a \check{\m f}_a\left({\m x},{\m x}^e\right)\right). 
	\end{split}  \label{eq:thm1_proof_3}
\end{align} 
From \eqref{eq:qb_condition}, the following inequalities are obtained
\begin{align}
	\begin{split}
	0&\leq \epsilon \check{\m x}_d^\top\bar{\m H}_d\check{\m x}_d - \epsilon \check{\m f}_d\left({\m x},{\m x}^e\right)^\top \check{\m f}_d\left({\m x},{\m x}^e\right) \\
	&\quad + \epsilon \check{\m x}_a^\top\bar{\m H}_a\check{\m x}_a - \epsilon \check{\m f}_a\left({\m x},{\m x}^e\right)^\top \check{\m f}_a\left({\m x},{\m x}^e\right), 
	\end{split}
	\label{eq:thm1_proof_4}
\end{align}
for a scalar $\epsilon\in \mbb{R}_{++}$. Next, adding \eqref{eq:thm1_proof_4} to the right-hand side of \eqref{eq:thm1_proof_3} yields the inequality
\begin{align}
	\dot{V}(t) \leq \m \omega^\top \m \Omega \m \omega, \label{eq:thm1_proof_5}
\end{align}
where $\m \omega := \bmat{\check{\m x}_d^\top \;\;\check{\m x}_a^\top \;\; \check{\m f}_d^\top\left({\m x},{\m x}^e\right)\;\; \check{\m f}_a^\top\left({\m x},{\m x}^e\right)}^\top$ and
\begin{align}
	\m \Omega := \bmat{ \m \Omega_{(1,1)}& * &*& * \\ {\m A}_a^\top \m P_2 & \m \Omega_{(2,2)} &*&*\\\m G_d^\top \m P_1 & \m O & -\epsilon\m I &* \\
		\m G_a^\top \m P_2 & \m G_a^\top \m P_3 & \m O & -\epsilon\m I }, \label{eq:thm1_proof_6}
\end{align}  
where the block diagonal matrices are specified as 
\begin{align*}
	\m \Omega_{(1,1)} &:= \bar{\m A}_d^\top \m P_1 + \m P_1^\top \bar{\m A}_d + \epsilon\bar{\m H}_d \\
	\m \Omega_{(2,2)} &:= {\m A}_a^\top \m P_3 + \m P_3^\top{\m A}_a+ \epsilon\bar{\m H}_a.
\end{align*}
It will be demonstrated in the sequel that the system of NDAEs \eqref{eq:nonlinearDAEpert} is asymptotically stable around the origin if $\m \omega^\top \m \Omega \m \omega < 0$ for any $\m \omega \neq 0$. Realize that this condition is equivalent to $ \m \Omega \prec 0$. By using Raleigh inequality, we have
\begin{align}
	\m \omega^\top \m \Omega \m \omega \leq \lambda_{\mathrm{max}}(\m \Omega)\norm{\m \omega}_2^2. \label{eq:thm1_proof_7}
\end{align}
Since the following also holds 
\begin{align*}
	\norm{\m \omega}_2^2 \leq (1+\lambda_{\mathrm{max}}(\bar{\m H}_d))\norm{\check{\m x}_d}_2^2 + (1+\lambda_{\mathrm{max}}(\bar{\m H}_a))\norm{\check{\m x}_a}_2^2, 
\end{align*}
thanks to \eqref{eq:qb_condition}, then from \eqref{eq:thm1_proof_7} one can simply obtain
\begin{align}
	\m \omega^\top \m \Omega \m \omega \leq -\eta_1   \norm{\check{\m x}_d}_2^2-\eta_2   \norm{\check{\m x}_a}_2^2, \label{eq:thm1_proof_9}
\end{align}
where in \eqref{eq:thm1_proof_9}, $\eta_1,\eta_2\in\mbb{R}_{++}$ defined as $\eta_1 := -\lambda_{\mathrm{max}}(\m \Omega)(1+\lambda_{\mathrm{max}}(\bar{\m H}_d))$ and $\eta_2 := -\lambda_{\mathrm{max}}(\m \Omega)(1+\lambda_{\mathrm{max}}(\bar{\m H}_a))$.
Now, as $\m P_1$ being nonsingular implies $$-\eta_1\norm{\check{\m x}_d}^2_2-\eta_2\norm{\check{\m x}_a}^2_2 \leq -\eta_1\lambda^{-1}_{\mathrm{max}}(\m E_d^\top \m P_1)V(t),$$ then \eqref{eq:thm1_proof_5} and \eqref{eq:thm1_proof_9} lead to
\begin{align}
	\dot{V}(t) &\leq -\eta_1\lambda^{-1}_{\mathrm{max}}(\m E_d^\top \m P_1)V(t) \nonumber\\
	\Rightarrow \int_{t_0}^{t}\dfrac{1}{V(\tau)} dV(\tau) &\leq \int_{t_0}^{t}-\eta_1 \lambda^{-1}_{\mathrm{max}}(\m E_d^\top \m P_1)\, d\tau \nonumber \\
	\Leftrightarrow V(t) &\leq e^{-\eta_1 \lambda^{-1}_{\mathrm{max}}(\m E_d^\top \m P_1)(t-t_0)} V(t_0).\label{eq:thm1_proof_10}
\end{align} 
Since $\norm{\check{\m x}_d}_2^2 \leq \lambda^{-1}_{\mathrm{min}}(\m E_d^\top \m P_1)V(t)$, then from \eqref{eq:thm1_proof_10} we obtain
\begin{align}
	\norm{\check{\m x}_d(t)}_2 &\leq \psi e^{-\tfrac{1}{2}{\eta_1}\lambda^{-1}_{\mathrm{max}}(\m E_d^\top \m P_1)(t-t_0)}\norm{\check{\m x}_d(t_0)}_2, \label{eq:thm1_proof_11}
\end{align}
where  $\psi> 0$ is a residual term given as $$\psi := \sqrt{\lambda_{\mathrm{min}}^{-1}\left(\m E_d^\top \m P_1\right)\lambda_{\mathrm{max}}\left(\m E_d^\top \m P_1\right)}.$$ The inequality \eqref{eq:thm1_proof_11} implies that $\norm{\check{\m x}_d(t)}_2\rightarrow 0$ as $t\rightarrow \infty$. 

\noindent	  \textit{(\textbf{b):}} Secondly, since we require $ \m \Omega \prec 0$, then it holds that the pair $\left(\mathrm{Blkdiag}(\m E_d,\m O),\mathrm{Blkdiag}(\bar{\m A}_d,{\m A}_a)\right)$
is both regular and impulse-free \cite{MASUBUCHI1997669}. As such, there exist nonsingular matrices  $\m M,\,\m N\in\mbb{R}^{n_x\times n_x}$ where $n_x := n_d + n_a$ such that \cite{duan2010analysis}
\begin{subequations}\label{eq:thm1_proof_12}
	\begin{align}
		\tilde{\m E} &= \m M \bmat{\m E_d & \m O \\ \m O & \m O} \m N = \bmat{\m I & \m O \\ \m O & \m O} \label{eq:thm1_proof_12-E}\\ \tilde{\m A} &= \m M \bmat{\bar{\m A}_d & \m O \\ \m O &{\m A}_a} \m N = \bmat{\tilde{\m A}_d & \m O \\ \m O &{\m I}},
	\end{align}
\end{subequations}
with $\m M,\,\m N$ partitioned as follows
$$ \m M = \bmat{\m M_1^\top & \m M_2^\top}^\top,\;\; \m N = \bmat{\m N_1 & \m N_2}, $$ where $\m M_1\in\mbb{R}^{n_d\times n_x}$, $\m M_2\in\mbb{R}^{n_a\times n_x}$, $\m N_1\in\mbb{R}^{n_x\times n_d}$, $\m N_2\in\mbb{R}^{n_x\times n_a}$.
In addition, define the transformed state $\tilde{\m x}\in\mbb{R}^{n_x}$ as
\begin{align}
	\tilde{\m x} = \bmat{\tilde{\m x}_d \\ \tilde{\m x}_a} := \m N^{-1}\bmat{\check{\m x}_d \\ \check{\m x}_a}, \;\;\tilde{\m x}_d\in\mbb{R}^{n_d},\;\tilde{\m x}_a\in\mbb{R}^{n_a}. \label{eq:thm1_proof_13}
\end{align}
It then can be directly shown the existence of matrices $\tilde{\m P}_1\in \mbb{R}^{n_d\times n_d}$, $\tilde{\m P}_2\in \mbb{R}^{n_d\times n_a}$, and $\tilde{\m P}_3\in \mbb{R}^{n_a\times n_a}$ such that 
\begin{align}
	\bmat{\tilde{\m P}_1& \m O \\ \tilde{\m P}_2 & \tilde{\m P}_3} = \m M^{-\top} \bmat{\m P_1 &\m O \\ \m P_2 & \m P_3} \m N, \label{eq:thm1_proof_14}
\end{align}
with $\tilde{\m P}_1$ is symmetric.
Since $V(t) = \check{\m x}_d^\top \m E_d^\top \m P_1 \check{\m x}_d = \tilde{\m x}_d^\top \tilde{\m P}_1 \tilde{\m x}_d$, then $\tilde{\m P}_1\succ 0$.
Using Schur complement, it is straightforward to show that $\m \Omega \prec 0$ is equivalent to $\tilde{\m \Omega} \prec 0$ where $\tilde{\m \Omega}$ is defined as
\begin{align*}
	\tilde{\m \Omega} := \tilde{\m A}^\top \tilde{\m P} + \tilde{\m P}^\top\tilde{\m A} + \epsilon \m N^\top \m \bar{\m H}^\top \m N + \epsilon \tilde{\m P}^\top \m M\m G\m G^\top \m M^\top \tilde{\m P},	
\end{align*}    
where $\bar{\m H} := \mathrm{Blkdiag}(\bar{\m H}_d,\bar{\m H}_a)$, ${\m G} := \mathrm{Blkdiag}(\bar{\m G}_d,\bar{\m G}_a)$ and $\tilde{\m P}$ is equal to the left-hand side of \eqref{eq:thm1_proof_14}. It can be shown from the $(2,2)$ block of $\tilde{\m \Omega}$ that $\tilde{\m \Omega} \prec 0$ implies $\tilde{\m P}_3 + \tilde{\m P}_3^{\top} \prec 0$. Now let us define a matrix measure function \cite{xu2006robust} $\nu:\mbb{R}^{n_a\times n_a}\rightarrow \mbb{R}$ as follows
\begin{align*}
	\nu \left(\tilde{\m P}_3\right) := \lim_{\theta \rightarrow 0^+} \frac{\norm{\m I + \theta \tilde{\m P}_3}_2 - 1}{\theta}.
\end{align*}
Due to Lemma 2.4 in \cite{xu2006robust}, then the inequality below holds
\begin{align}
	\lambda_{\mathrm{max}}\left(\tilde{\m P}_3\right) \leq \nu \left(\tilde{\m P}_3\right)  = \frac{1}{2} \lambda_{\mathrm{max}} \left(\tilde{\m P}_3 + \tilde{\m P}_3^{\top}\right). \label{eq:thm1_proof_18}
\end{align} 
The above inequality suggests that $\tilde{\m P}_3$ is nonsingular as  $\tilde{\m P}_3 + \tilde{\m P}_3^{\top} \prec 0$ infers that the right-hand side of \eqref{eq:thm1_proof_18} is negative. This shows that the matrix $\tilde{\m P}$ defined in \eqref{eq:thm1_proof_14} is nonsingular. However, since $\m M,\m N$ are also nonsingular, then it can be inferred from \eqref{eq:thm1_proof_14} that $\m P_1$ and $\m P_3$ are nonsingular---this confirms the validity of the previous assumption.

\noindent \textbf{\textit{(c):}} Thirdly, from the fact that the $(2,2)$ block of $\tilde{\m \Omega}$ is negative definite, then for a constant $\delta > 0$, we have
\begin{align}
	\tilde{\m P}_3 + \tilde{\m P}_3^\top +\epsilon \m N_2^\top \bar{\m H} \m N_2 + \tilde{\m P}_3^\top \m \Xi \tilde{\m P}_3 \prec 0,\label{eq:thm1_proof_19}
\end{align}
where $\m \Xi := \epsilon \m M_2 \m G \m G^\top \m M_2 + \delta \m I$ is nonsingular. Notice that \eqref{eq:thm1_proof_19} can be written as \cite{Lu2006observer}
\begin{align*}
	\left(\tilde{\m P}_3+\m \Xi^{-1}\right)^\top\m \Xi	\left(\tilde{\m P}_3+\m \Xi^{-1}\right) - \m \Xi^{-1} + \epsilon \m N_2^\top \bar{\m H} \m N_2 \prec 0.
\end{align*} 
Since we have $\left(\tilde{\m P}_3+\m \Xi^{-1}\right)^\top\m \Xi	\left(\tilde{\m P}_3+\m \Xi^{-1}\right) \succ 0$, then from the above equation, there exists $\zeta > 0$ such that \cite{Lu2006observer}
\begin{align}
	(\epsilon + \zeta)\m N_2^\top \bar{\m H} \m N_2 - \m \Xi^{-1} \prec 0. \label{eq:thm1_proof_20}
\end{align}
It then can be shown from \eqref{eq:thm1_proof_20} and Lemma \ref{lem:lemma1} that
\begin{align}
	&\norm{\bar{\m H}^{\frac{1}{2}}\m N_2 \m M_2 \m G \check{\m f}({\m x},{\m x}^e)}_2^2  \nonumber \\
	&\qquad \leq \frac{1}{\epsilon + \zeta} \check{\m f}^\top({\m x},{\m x}^e)\m G^\top \m M_2^\top \m \Xi^{-1} \m M_2 \m G \check{\m f}({\m x},{\m x}^e)  \nonumber\\
	&\qquad \leq \frac{\epsilon}{\epsilon + \zeta} \check{\m f}^\top({\m x},{\m x}^e)\check{\m f}({\m x},{\m x}^e),\nonumber
\end{align}
where $\check{\m f}(\check{\m x}_d,\check{\m x}_a) := \bmat{\check{\m f}_d^\top({\m x},{\m x}^e)\,\,\check{\m f}_a^\top({\m x},{\m x}^e)}^\top$, implying
\begin{align}
	\norm{\check{\m f}({\m x},{\m x}^e)}^2_2 \leq \frac{\epsilon + \zeta}{\zeta} \norm{\bar{\m H}^{\frac{1}{2}}\m N_2}^2_F\norm{\tilde{\m x}_d}_2^2.\label{eq:thm1_proof_21}
\end{align}
Using \eqref{eq:thm1_proof_21}, it is straightforward to show that
\begin{align*}
	\norm{\tilde{\m x}_a}_2 \leq \sqrt{\tfrac{\epsilon + \zeta}{\zeta}}\norm{\m M_2 \m G}_F\norm{\bar{\m H}^{\frac{1}{2}}\m N_2}_F\norm{\tilde{\m x}_d}_2,
\end{align*}
which, according to \eqref{eq:thm1_proof_11}, leads to
\begin{align}
	\norm{\check{\m x}_a(t)}_2 &\leq \varrho e^{-\tfrac{1}{2}{\eta_1}\lambda^{-1}_{\mathrm{max}}(\m E_d^\top \m P_1)(t-t_0)}\norm{\check{\m x}_d(t_0)}_2, \label{eq:thm1_proof_22}
\end{align}
where $\varrho > 0$ is a  residual term. The inequality \eqref{eq:thm1_proof_22} indicates that $\norm{\check{\m x}_a(t)}_2\rightarrow 0$ as $t\rightarrow \infty$. 

\noindent \textbf{\textit{(d):} }Finally, since $\m P_1$ and $\m P_3$ are nonsingular, we can define ${\m Q}_1\in \mbb{R}^{n_d\times n_d}$, ${\m Q}_2\in \mbb{R}^{n_d\times n_a}$, and ${\m Q}_3\in \mbb{R}^{n_a\times n_a}$ such that
\begin{align}
	{\m Q}_1 := \m P_1^{-1},\;{\m Q}_2 := -\m P_3^{-1}\m P_2\m P_1^{-1},\;{\m Q}_3 := \m P_3^{-1}. \label{eq:thm1_proof_23}
\end{align}
Using congruence transformation, given the new matrices defined in \eqref{eq:thm1_proof_23}, and applying the Schur complement, the condition $\m \Omega \prec 0$ can be shown equivalent to \eqref{eq:LMI_stabilization} where $\bar{\epsilon} := \frac{1}{\epsilon}$. Note that substituting ${\m Q}_1 = \m P_1^{-1}$ to  $\m E_d^\top \m P_1 =\m P_1^\top\m E_d \succ 0$ establishes \eqref{eq:LMI_stabilization_Q}. This completes the proof.
\newqed

\vspace{-0.1cm}
\section{Proof of Proposition \ref{prs:LMI_stability}}\label{appdx:C}
	Notice that, from \eqref{eq:thm1_proof_12-E} and \eqref{eq:thm1_proof_14}, we have
	\begin{align*}
		\bmat{\m E_d & \m O \\ \m O & \m O}^\top &= \m N^{-\top} \bmat{\m I & \m O \\ \m O & \m O} \m M^{-\top} \\
		\bmat{\m P_1 &\m O \\ \m P_2 & \m P_3}^{-1} &= \m N \bmat{\m \Pi_1 & \m O \\ \m \Pi_2 & \m \Pi_3}\m M^{-\top},
	\end{align*}
	where $\m \Pi_1 := \tilde{\m P}_1^{-1}$, $\m \Pi_2 :=- \tilde{\m P}_3^{-1}\tilde{\m P}_2\tilde{\m P}_1^{-1}$, and $\m \Pi_3 := \tilde{\m P}_3^{-1}$. The second equation can be written as 
	\begin{align}
		\begin{split}
			\bmat{\m P_1 &\m O \\ \m P_2 & \m P_3}^{-1} &= \m N\bmat{\m \Pi_1 & \m O \\ \m O & \m I}\bmat{\m I & \m O \\ \m O & \m O}\m M^{-\top} \\
			&\quad + \m N \bmat{\m O \\ \m I}\bmat{\m \Pi_2 & \m \Pi_3}\m M^{-\top}.
		\end{split}\label{eq:thm2_proof_1}
	\end{align} 
	Since $\m N^\top \mathrm{Blkdiag}\left(\m E_d^\top,\m O\right)\m M^\top \bmat{\m O\;\;\m I} = 0$, then there exists a full rank matrix $\m \Phi\in\mbb{R}^{n_a\times n_a}$ such that \cite{Lu2003stability}
	\begin{align*}
		\m \Phi \bmat{\m O\;\;\m I}\m N^\top \bmat{\m E_d & \m O \\ \m O & \m O}^\top = 0,
	\end{align*}
	which allows \eqref{eq:thm2_proof_1} to be expressed as
	\begingroup
	\allowdisplaybreaks
	\begin{align*}
			\bmat{\m P_1 &\m O \\ \m P_2 & \m P_3}^{-1} &= \m N\bmat{\m \Pi_1 & \m O \\ \m O & \m I}\m N^\top\m N^{-\top}\bmat{\m I & \m O \\ \m O & \m O}\m M^{-\top} \\
			&\quad+ \m N \bmat{\m O \\ \m I}\m \Phi^\top\m \Phi^{-\top}\bmat{\m \Pi_2 & \m \Pi_3}\m M^{-\top}.
	\end{align*}
\endgroup
	Following \cite{Lu2003stability}, it is not difficult to show that the above ensures the existence of matrices $\m X_1\in\mbb{S}^{n_d}_{++}$, $\m X_2\in\mbb{R}^{n_a\times n_d}$, $\m R\in\mbb{R}^{n_a\times n_a}$, and $\m Y\in\mbb{R}^{n_d\times n_a}$ such that  
	\begin{align}
		\m Q_1 = \m X_1 \m E_d^\top,\;\;\m Q_2 = \m X_2 \m E_d^\top + \m Y,\;\; \m Q_3 = \m R. \label{eq:thm2_proof_3}
	\end{align}         
	At last, by substituting \eqref{eq:thm2_proof_3} to \eqref{eq:LMI_stabilization} and defining $\m W := \m K_d\m X_1$ for a matrix $\m W\in\mbb{R}^{n_u\times n_d}$ \eqref{eq:LMI_stabilization_linear} is established. Since \eqref{eq:thm2_proof_3} indeed satisfies \eqref{eq:LMI_stabilization_Q}, then we are done.
	\newqed                                

\end{document}